\documentclass[11pt,a4paper]{article}

\usepackage{amsmath}
\usepackage{amssymb}
\usepackage{stmaryrd}
\makeatletter
\usepackage[margin=1in]{geometry}
\usepackage{hyperref}
\usepackage{mathtools}
\usepackage{amsthm}
\usepackage{upquote}
\usepackage{amsfonts}
\usepackage{amssymb}\usepackage{amsmath}
\usepackage{amsfonts}
\usepackage{amssymb}
\usepackage{bbm}
\usepackage{enumitem}
\usepackage{tikz}
\usetikzlibrary{decorations.markings}
\usepackage{tikz-cd}
\usepackage{authblk}
\usetikzlibrary{babel}

\newtheorem{theorem}{Theorem}[section]
\newtheorem{lemma}[theorem]{Lemma}
\newtheorem{proposition}[theorem]{Proposition}
\newtheorem{corollary}[theorem]{Corollary}
\newtheorem{conjecture}[theorem]{Conjecture}

\newtheorem{definition}{Definition}[section]

\AtEndDocument{
\textsc{Department of Mathematics, UCLA, Los Angeles, CA 90095} \\
\textsc{Email:} \texttt{eprimozic@math.ucla.edu}}

\begin{document}
\title{Motivic Steenrod operations in characteristic $p$}
\author{Eric Primozic}

\maketitle
\setcounter{secnumdepth}{1}
\section*{Introduction}
\addcontentsline{toc}{section}{Introduction}
\ 

\indent Voevodsky constructed motivic reduced power operations $P^{n}_{F}$ for $n \geq0$ where the base field $F$ is a perfect field with $\mathrm{char}(F)$ not equal to the characteristic $p>0$ of the coefficient field \cite{Voe}. These operations were used in the proof of the Bloch-Kato conejcture. Brosnan gave an elementary construction of Steenrod operations on mod $p$ Chow groups over a base field of characteristic $\neq p$ \cite{Bro}. Steenrod operations on Chow groups have been used succesfully in the study of quadratic forms over a base field of characteristic $\neq 2$ and to prove degree formulas in algebraic geometry as in \cite{EKM} and \cite{Mer}.

\indent For a prime $p$, Voevodsky's construction of Steenrod operations for the coefficient field $\mathbb{F}_{p}$ uses the calculation of the motivic cohomology of $BS_{p}$. However, when defined over a base field $k$ of characteristic $p$, $B\mathbb{Z}/p$ is contractible \cite[Proposition 3.3]{MorVoe}. Hence, over the base field $k$, $H^{*,*}(BS_{p}, \mathbb{F}_{p})\cong H^{*,*}(k, \mathbb{F}_{p})$ and so one cannot carry out Voevodsky's construction. It has also been an open problem to just define Steenrod operations on the mod $p$ Chow groups of smooth schemes over a field of characteristic $p$. Haution made progress on this problem by constructing the first $p-1$ homological Steenrod operations on Chow groups mod $p$ and $p$-primary torsion over any base field \cite{Hau4}, defining the first Steenrod square on mod $2$ Chow groups over any base field \cite{Hau2}, and constructing weak forms of the second and third Steenrod squares over a field of characteristic $2$ \cite{Hau5}. Note that in papers where Steenrod squares (or weak forms of Steenrod squares) on mod $2$ Chow groups are used, the $n$th Steenrod square on mod $2$ Chow groups corresponds to the $2n$th Steenrod square on mod $2$ motivic cohomology since the Bockstein homomorphism is $0$ on mod $2$ Chow groups.

\indent For $p$ a prime,  we use the results of Frankland and Spitzweck in \cite{FrankSpi} to define Steenrod operations $P_{k}^{n}: H^{i,j}(-, \mathbb{F}_p) \to H^{i+2n(p-1),j+n(p-1)}(-, \mathbb{F}_p) $ for $n \geq 0$ on the mod $p$ motivic cohomology of smooth schemes over a field $k$ of characteristic $p$. Note that some authors use the notation $H^{i}(-,\mathbb{Z}(j))$ in place of $H^{i,j}(-, \mathbb{Z})$ to denote motivic cohomology. For $ n \geq 1$, we show that $P^{n}_{k}$ is the $p$th power on $H^{2n,n}(-,\mathbb{F}_{p})=CH^{n}(-)/p$, and we also prove an instability result for the Steenrod operations. Restricted to mod $p$ Chow groups, we prove that the $P_{k}^i$ satisfy expected properties such as Adem relations and the Cartan formula. We also show that the operations $P^{n}_{k}$ agree with the operations $P^{n}_{K}$, constructed by Voevodsky for $\mathrm{char}(K)=0$, on the mod $p$ Chow rings of flag varieties in characteristic $0$.

\indent In Section \ref{Rost}, we extend Rost's degree formula \cite[Theorem 6.4]{Mer}  to a base field of arbitrary characteristic. The degree formula we obtain at odd primes seems to be new. In Section \ref{S:appquad}, we use the new operations to study quadratic forms defined over a field of characteristic $2$. Previous results or proofs avoided the case of quadratic forms in characteristic $2$ since Steenrod squares were not available. We prove Hoffmann's conjecture (a generalization including characteristic $2$ quadratic forms) on the possible values of the first Witt index of anisotropic quadratic forms for the case of nonsingular anisotropic quadratic forms over a field of characteristic $2$. In characteristic $\neq 2$, Hoffmann's conjecture was proved by Karpenko in \cite{Kar}. Previously, Haution used a weak form of the first homological Steenrod square to prove a result on the parity of the first Witt index for nonsingular anisotropic quadratic forms over a field of characteristic $2$ \cite[Theorem 6.2]{Hau}. Haution's result is a corollary of the case of Hoffmann's conjecture proved in this paper. Using the Steenrod squares defined in this paper, it should be possible to extend other results on quadratic forms to the case where the base field has characteristic $2$.

\section*{Acknowledgments} I thank Burt Totaro for suggesting this project to me and for providing advice. I am very grateful to Marc Hoyois for answering my questions and telling me the strategy used in the proof of Proposition \ref{T:pullbaclEB} along with most of the details of the proof. I thank Markus Spitzweck for answering my questions. I thank Chuck Weibel for his comments. I thank Nikita Karpenko for telling me about some of the applications of the new Steenrod operations to quadratic forms in characteristic $2$ given in Section \ref{S:appquad}.

\section{Prior results on the dual Steenrod algebra and setup}

\indent Let $k$ be a field of characteristic $p>0$. For a base scheme $S$, we let $\textup{Sm}_{S}$ denote the category of quasi-projective separated smooth schemes of finite type over $S$, let $H(S)$ denote the unstable motivic homotopy category of spaces over $S$ defined by Morel-Voevodsky \cite{MorVoe}, let $H_{\bullet}(S)$ denote the pointed unstable motivic homotopy category of spaces over $S$, and we let $SH(S)$ denote the stable motivic homotopy category of spectra over $S$ \cite{Voe3}. Let $$\Sigma^{\infty}_{+}:\textup{Sm}_{S} \to SH(S),$$ 
$$\Sigma^{\infty}_{+}:H(S) \to H_{\bullet}(S) \to SH(S)$$ denote the infinite $\mathbb{P}^{1}$-suspension functors. 

\indent We recall some results from \cite{Spi} and \cite{FrankSpi} that hold in the categories $H(k)$ and $SH(k)$. Let $B\mu_{p} \in H(k)$ denote the geometric motivic classifying space of the group scheme $\mu_{p}$ over $k$ of the $p$th roots of unity. Let $H\mathbb{F}^{k}_{p} \in SH(k)$ denote the motivic Eilenberg-MacLane spectrum representing mod $p$ motivic cohomology. Let $v \in H^{2,1}(B\mu_{p}, \mathbb{F}_p)$ denote the pullback of the first Chern class $c_{1}\in H^{2,1}(B\mathbb{G}_{m}, \mathbb{F}_p)$. From the computation of the motivic cohomology of $B\mu_{p}$ in \cite[Theorem 6.10]{Voe}, there exists a unique $u \in H^{1,1}(B\mu_{p}, \mathbb{F}_p)$ such that $\beta(u)=v$ where $\beta$ denotes the Bockstein homomorphism on mod $p$ motivic cohomology. The class  $\rho=-1$ in $H^{1,1}(k, \mathbb{F}_{p})=k^{*}/k^{* \, p}$ is $0$ and the class of $\tau \in H^{0,1}(k, \mathbb{F}_{p})=\mu_{p}(k)=0$ described in \cite[Theorem 6.10]{Voe} is also $0$. We need the following computation which can be deduced from \cite[Theorem 6.10]{Voe} by setting $\rho=0$ and $\tau=0$. 

\begin{theorem} \label{T:comp of coh bmu}
There is an isomorphism $$H^{*,*}(B\mu_{p}, \mathbb{F}_{p}) \cong H^{*,*}(k, \mathbb{F}_{p})\llbracket v,u\rrbracket /(u^{2}).$$
\end{theorem}
Note that $H^{*,*}(B\mu_{p}, \mathbb{F}_{p})$ is defined in \cite{Voe} as a limit of motivic cohomology rings of smooth schemes over the base field. This explains why power series appear in the above theorem.

\indent Let $\mathcal{A}^{k}_{*,*} \coloneqq \pi_{*,*}(H\mathbb{F}^{k}_{p} \wedge H\mathbb{F}^{k}_{p})$. As described in \cite[Chapter 10.2]{Spi}, there is a coaction map \begin{equation} \label{eq:coaction} H^{*,*}(B\mu_{p}, \mathbb{F}_p) \to \mathcal{A}^{k}_{-*,-*} \widehat{\otimes}_{\pi_{-*,-*}H\mathbb{F}_{p}^{k}}H^{*,*}(B\mu_{p}, \mathbb{F}_p). \end{equation} We use the left $H\mathbb{F}_{p}^{k}$-module structure on $H\mathbb{F}_{p}^{k}\wedge H\mathbb{F}_{p}^{k}$ for this coaction map. For $i\geq 0$ and $j \geq 1$, classes $\tau_{i} \in \mathcal{A}^{k}_{2p^{i}-1,p^{i}-1} $ and $\xi_{j} \in \mathcal{A}^{k}_{2p^{j}-2,p^{j}-1}$ are defined by the coaction map: $$     
 u \mapsto u+ \Sigma _{i\geq 0}\tau_{i} \otimes v^{p^{i}},$$
 
$$v \mapsto v+ \Sigma _{j\geq 1}\xi_{j} \otimes v^{p^{j}}.$$

\begin{proposition} \label{P:tau2=0} $\tau_{i}^{2}=0$ for all $i \geq 0$.
\end{proposition}
\begin{proof} We use the argument of \cite[Theorem 12.6]{Voe}. First, we assume that $\mathrm{char}(k)=2$. Under the coaction map \ref{eq:coaction}, $$u^{2}=0 \mapsto u^{2}+\Sigma _{i\geq 0}\tau^{2}_{i} \otimes v^{2^{i+1}}=0.$$ For $i \geq0$, the coefficient of $v^{2^{i+1}}$ equals $0=\tau^{2}_{i}$.

\indent Now we assume that $p=\mathrm{char}(k)$ is odd. Let $i\geq0$. As  $\mathcal{A}^{k}_{*,*}$ is graded-commutative under the first grading, we have $\tau^{2}_{i}=(-1)^{(2p^{i}-1)(2p^{i}-1)}\tau^{2}_{i}=-\tau^{2}_{i}$ which implies that $\tau^{2}_{i}=0$. 
\end{proof}
\indent In this paper, we shall consider finite sequences $\alpha=(\epsilon_{0},r_{1},\epsilon_{1},r_{2},\ldots )$ of integers such that $\epsilon_{i}\in \{0,1\}$ and $r_{j} \geq 0$ for all $i \geq0$ and $j\geq 1$. From now on, it will be assumed that any sequence $\alpha$ in this paper satisfies these conditions. To a sequence $\alpha$, we associate a monomial $\omega(\alpha)=\tau_{0}^{\epsilon_{0}}\xi_{1}^{r_1}\tau_{1}^{\epsilon_{1}} \cdots \in \mathcal{A}^{k}_{*,*}$ of bidegree $(p_{\alpha},q_{\alpha})$. The sequences $\alpha$ induce a morphism $$\Psi_{k}: \bigoplus\limits _{\alpha} \Sigma ^{p_{\alpha},q_{\alpha}}H\mathbb{F}^{k}_{p} \to H\mathbb{F}^{k}_{p} \wedge H\mathbb{F}^{k}_{p}$$ of left $H\mathbb{F}^{k}_{p}$-modules. Frankland and Spitzweck proved the following theorem \cite[Theorem 1.1]{FrankSpi} which allows us to define Steenrod operations on mod $p$ motivic cohomology over the base $k$.

\begin{theorem} \label{T:Psidef}
The morphism $$\Psi_{k}: \bigoplus\limits _{\alpha} \Sigma ^{p_{\alpha},q_{\alpha}}H\mathbb{F}^{k}_{p} \to H\mathbb{F}^{k}_{p} \wedge H\mathbb{F}^{k}_{p}$$ is a split monomorphism of left $H\mathbb{F}^{k}_{p}$-modules.
\end{theorem}

\indent  It is conjectured that $\Psi_{k}$ is an isomorphism. Frankland and Spitzweck proved this theorem by comparing $\Psi_{k}$ to the corresponding isomorphism \begin{equation} \label{eq:char0dual}\Psi_{K}: \bigoplus\limits _{\alpha} \Sigma ^{p_{\alpha},q_{\alpha}}H\mathbb{F}^{K}_{p} \to H\mathbb{F}^{K}_{p} \wedge H\mathbb{F}^{K}_{p} \end{equation} of left $H\mathbb{F}^{K}_{p}$-modules for $\mathrm{char}(K)=0$. From now on, we will identify $\bigoplus\limits _{\alpha} \Sigma ^{p_{\alpha},q_{\alpha}}H\mathbb{F}^{K}_{p}$ with $H\mathbb{F}^{K}_{p} \wedge H\mathbb{F}^{K}_{p}$ as left $H\mathbb{F}_{p}^{K}$-modules through $\Psi_{K}$ whenever $K$ is a field of characteristic $0$. Let $D$ be a complete unramified discrete valuation ring with closed point $i:\mathrm{Spec}(k) \to \mathrm{Spec}(D)$ and generic point $j:\mathrm{Spec}(K) \to \mathrm{Spec}(D)$ where $K=\mathrm{Frac}(D)$. For example, when $k=\mathbb{F}_{p}$, we take $D=\mathbb{Z}_{p}$ and $K=\mathbb{Q}_{p}$.

\indent For a morphism $f: S_{1} \to S_{2}$ of base schemes, we let $f_{*}\coloneqq Rf_{*}: SH(S_{1}) \to SH(S_{2})$ and $f^{*} \coloneqq Lf^{*}:SH(S_{2}) \to SH(S_{1})$ denote the right derived pushforward and left derived pullback functors respectively. Pullback $f^{*}$ is strongly monoidal while $f_{*}$ is lax monoidal. Furthermore, $f_{*}$ commutes with all suspensions $\Sigma^{i,j}$ \cite[Lemma 7.5]{FrankSpi}. We also note that $f_{*}$ preserves coproducts \cite[Lemma 7.4]{FrankSpi}.

\indent For $S$ a separated Noetherian scheme of finite dimension, we let $\widehat{H}\mathbb{Z}^{S}\in SH(S)$ denote the motivic $E_{\infty}$ ring spectrum constructed by Spitzweck in \cite{Spi} and let $\widehat{H}\mathbb{F}^{S}_{p}\coloneqq \widehat{H}\mathbb{Z}^{S}/p$ . Let $D(\widehat{H}\mathbb{Z}^{S})$ denote the homotopy category of left $\widehat{H}\mathbb{Z}^{S}$-modules. See \cite[Section 7.2]{CisDeg} and \cite[Sections 2 and 3]{FrankSpi} for a discussion on the homotopy category of left $R$-modules $D(R)$ for a highly structured ring spectrum $R$. There is a forgetful functor $U_{S}:D(\widehat{H}\mathbb{Z}^{S}) \to SH(S)$.

\indent The spectrum $\widehat{H}\mathbb{Z}^{S}$ enjoys a number of desirable properties. The spectrum $\widehat{H}\mathbb{Z}^{S}$ is Cartesian. This means that for a morphism $f:S_{1} \to S_{2}$ of base schemes, the induced morphism $f^{*}\widehat{H}\mathbb{Z}^{S_{2}} \to \widehat{H}\mathbb{Z}^{S_{1}}$ is an isomorphism in $SH(S_{1})$ of $E_{\infty}$ ring spectra \cite[Chapter 9]{Spi}. Throughout this paper, we will frequently identify $f^{*}\widehat{H}\mathbb{Z}^{S_{2}}$ with $\widehat{H}\mathbb{Z}^{S_{1}}$ whenever we are given a morphism $f:S_{1} \to S_{2}$ of base schemes. See also \cite[Section 2]{FrankSpi}. Hence, the square 
\[
\begin{tikzcd}
D(\widehat{H}\mathbb{Z}^{S_{2}}) \arrow[r, "f^{*}"] \arrow[d,"U_{S_{2}}"] 
& D(\widehat{H}\mathbb{Z}^{S_{1}}) \arrow[d, "U_{S_{1}}"]\\
SH(S_{2}) \arrow[r, "f^{*}"] & SH(S_{1})
\end{tikzcd}
\]
commutes.

\indent For $S=\mathrm{Spec}(F)$ with $F$ a field, $\widehat{H}\mathbb{Z}^{S}$ is isomorphic as an $E_{\infty}$ ring spectrum to the usual Eilenberg-MacLane spectrum $H\mathbb{Z}^{S}$ constructed by Voevodsky \cite[Theorem 6.7]{Spi}. For the discrete valuation ring $D$,  $\widehat{H}\mathbb{Z}^{D}$ represents Bloch-Levine motivic cohomology as defined in \cite{Lev}. 

\indent We briefly describe the definition of Bloch-Levine motivic cohomology in \cite{Lev} for a discrete valuation ring $D$. Let $X \to \mathrm{Spec}(D)$ be a morphism of finite type with $X$ irreducible. If the image of the generic point $\eta_{X}$ of $X$ is $\mathrm{Spec}(k)$, then we define dim$(X) \coloneqq$dim$(X_{\mathrm{Spec}(k)})$. Otherwise, we define dim$(X) \coloneqq$dim$(X_{\mathrm{Spec}(K)})+1$. For $n\geq 0$, let $\Delta^{n} \coloneqq \mathrm{Spec}(D[t_{0}, \ldots, t_{n}]/ \Sigma_{i}t_{i}-1)$ denote the algebraic $n$-simplex over $D$. Let $z_{q}(X, r)$ denote the free abelian group generated by all irreducible closed subschemes $C \subset \Delta^{r} \times _{\mathrm{Spec}(D)} X$ of dimension $r+q$ such that $C$ meets each face of $\Delta^{r} \times _{\mathrm{Spec}(D)} X$ properly. We then set $z^{q}(X, r)=z_{\mathrm{dim}(X)-q}(X, r)$ so that we get a pullback homomorphism $z^{q}(X, r) \to z^{q}(X, r-1)$ for each face of $\Delta^{r}$. Then the Zariski hypercohomology of the complex $z^{q}(X, *)$ with alternating face maps is Bloch-Levine motivic cohomology (with the appropriate shift).

\begin{theorem} \label{T:defofpi0andpi}

\item The morphism $H\mathbb{F}^{k}_{p}\cong i^{*}(\widehat{H}\mathbb{F}^{D}_{p}) \to i^{*}j_{*}H\mathbb{F}^{K}_{p} \cong i^{*}j_{*}j^{*}\widehat{H}\mathbb{F}^{D}_{p} $ in $D(H\mathbb{F}^{k}_{p})$ induced by adjunction induces a splitting $i^{*}j_{*}H\mathbb{F}^{K}_{p}\cong H\mathbb{F}^{k}_{p} \oplus \Sigma^{-1,-1}H\mathbb{F}^{k}_{p}$ in $D(H\mathbb{F}^{k}_{p})$. We let $\pi:i^{*}j_{*}H\mathbb{F}^{K}_{p} \to H\mathbb{F}^{k}_{p}$ and $\pi_{0}:i^{*}j_{*}H\mathbb{F}^{K}_{p} \to \Sigma^{-1,-1}H\mathbb{F}^{k}_{p}$  denote the projections induced by this splitting. There is also a splitting $i^{*}j_{*}H\mathbb{Z}^{K}\cong H\mathbb{Z}^{k} \oplus \Sigma^{-1,-,1}H\mathbb{Z}^{k}$ in $D(H\mathbb{Z}^{k})$ \cite[Lemma 4.10]{FrankSpi}.

\end{theorem}

\indent Let $\eta: id. \to j_{*}j^{*}$ denote the unit map. From now on, we shall denote all adjunction morphisms $i^{*}E \to i^{*}j_{*}j^{*}E$ for $E \in SH(D)$ by $i^{*}\eta$. We will also denote all $\Sigma^{s,t}\pi, \Sigma^{s,t}\pi_{0}$ by $\pi$ and $\pi_{0}$ respectively to make the text easier to read. The morphisms $\Psi_{k}$ and $\Psi_{K}$ lift to a morphism $$\Psi_{D}:\bigoplus\limits _{\alpha} \Sigma ^{p_{\alpha},q_{\alpha}}\widehat{H}\mathbb{F}^{D}_{p} \to \widehat{H}\mathbb{F}^{D}_{p} \wedge \widehat{H}\mathbb{F}^{D}_{p}$$ in $D(\widehat{H}\mathbb{F}^{D}_{p})$ \cite[Lemma 3.10]{FrankSpi}. Applying $i^{*}\eta$ to $\Psi_{D}$, we get a commuting square
\begin{equation}              \label{PsiDcommute}
\begin{tikzcd}
\bigoplus\limits _{\alpha} \Sigma ^{p_{\alpha},q_{\alpha}}H\mathbb{F}^{k}_{p} \arrow[r, "\Psi_{k}"] \arrow[d, "i^{*}\eta"]
& H\mathbb{F}^{k}_{p}\wedge H\mathbb{F}^{k}_{p} \arrow[d, "i^{*}\eta"] \\
\bigoplus\limits _{\alpha} \Sigma ^{p_{\alpha},q_{\alpha}}i^{*}j_{*}H\mathbb{F}^{K}_{p} \arrow[r, "i^{*}j_{*}\Psi_{K}"]
& i^{*}j_{*}(H\mathbb{F}^{K}_{p}\wedge H\mathbb{F}^{K}_{p})
\end{tikzcd}
\end{equation}
in $D(H\mathbb{F}^{k}_{p})$. Let $r:H\mathbb{F}^{k}_{p}\wedge H\mathbb{F}^{k}_{p} \to \bigoplus\limits _{\alpha} \Sigma ^{p_{\alpha},q_{\alpha}}H\mathbb{F}^{k}_{p}$ be the retraction of $\Psi_{k}$ defined by the following composite \cite[Theorem 5.1]{FrankSpi}. 
\[
\begin{tikzcd}
H\mathbb{F}^{k}_{p}\wedge H\mathbb{F}^{k}_{p} \arrow[r, "i^{*}\eta"]
& i^{*}j_{*}(H\mathbb{F}^{K}_{p}\wedge H\mathbb{F}^{K}_{p}) \arrow[r, "i^{*}j_{*}\Psi^{-1}_{K}"]
& \bigoplus\limits _{\alpha} \Sigma ^{p_{\alpha},q_{\alpha}}i^{*}j_{*}H\mathbb{F}^{K}_{p} \arrow[d, "\oplus \pi"] \\
& & \bigoplus\limits _{\alpha} \Sigma ^{p_{\alpha},q_{\alpha}}H\mathbb{F}^{k}_{p}
\end{tikzcd}
\]

\indent For $S=k, K,$ or $D$ (use $\widehat{H}\mathbb{F}^{D}_{p}$), we let $\mu^{S}_{1}: H\mathbb{F}^{S}_{p} \wedge H\mathbb{F}^{S}_{p} \to H\mathbb{F}^{S}_{p}$ denote the multiplication morphism. There is also a multiplication morphism $$\mu^{S}_{2}:(H\mathbb{F}^{S}_{p} \wedge H\mathbb{F}^{S}_{p})\wedge (H\mathbb{F}^{S}_{p} \wedge H\mathbb{F}^{S}_{p}) \to H\mathbb{F}^{S}_{p} \wedge H\mathbb{F}^{S}_{p}$$ defined in the standard way by interchanging the two middle $H\mathbb{F}^{S}_{p}$ terms and then applying $\mu^{S}_{1} \wedge \mu^{S}_{1}$.

\indent For a sequence $\alpha_{0}$, we define $i^{*}\eta_{\alpha_{0}}: H\mathbb{F}^{k}_{p} \wedge H\mathbb{F}^{k}_{p} \to \Sigma ^{p_{\alpha_{0}},q_{\alpha_{0}}}H\mathbb{F}^{k}_{p}$ in $D(H\mathbb{F}^{k}_{p})$ to be the composite 
\begin{equation} \label{eq:retractexplicit r}
\begin{tikzcd}
H\mathbb{F}^{k}_{p} \wedge H\mathbb{F}^{k}_{p} \arrow[r, "i^{*}\eta"] 
& i^{*}j_{*}(H\mathbb{F}^{K}_{p} \wedge H\mathbb{F}^{K}_{p}) \arrow[r, "i^{*}j_{*}\Psi^{-1}_{K}"]
                & \bigoplus\limits _{\alpha} \Sigma ^{p_{\alpha},q_{\alpha}}i^{*}j_{*}H\mathbb{F}^{k}_{p} \arrow[r, "proj."] 
                & \Sigma^{p_{\alpha_{0}},q_{\alpha_{0}}}i^{*}j_{*}H\mathbb{F}^{K}_{p} \arrow[d, "\pi"]\\
& & &  \Sigma^{p_{\alpha_{0}},q_{\alpha_{0}}}H\mathbb{F}^{k}_{p}.
\end{tikzcd}
\end{equation}  
The morphism $i^{*}\eta_{\alpha_{0}}$ is a retract of the morphism $H\mathbb{F}^{k}_{p} \wedge \omega(\alpha_{0}):\Sigma^{p_{\alpha_{0}},q_{\alpha_{0}}}H\mathbb{F}^{k}_{p} \to H\mathbb{F}^{k}_{p} \wedge H\mathbb{F}^{k}_{p}$.

\indent From the work of Voevodsky \cite{Voe2} and Friedlander-Suslin \cite[Corollary 12.2]{FriSus}, Bloch's higher Chow groups are isomorphic to motivic cohomology as defined by Voevodsky. The isomorphism between motivic cohomology and Bloch's higher Chow groups is compatible with pullback maps and product structures \cite[Theorem 6.7]{Spi}. See also \cite{KonYas}.

\begin{theorem} \label{voesusfri}
Let $F$ be a field and let $X \in \textup{Sm}_{F}$. Then $$H^{n,i}(X, \mathbb{Z}) \cong CH^{i}(X,2i-n)$$ for all $n$ and $i \geq 0$.

\end{theorem}

\indent Let $n,i \geq 0$ such that $ n>2i$. From the above theorem, we get that $H^{n,i}(X,A)=0$ for any coefficient ring $A$ and $X \in \textup{Sm}_{F}$. 
\section{Definition of operations} \label{sec:def}
\indent In this section, we use the results of Frankland and Spitzweck in \cite{FrankSpi} to define new Steenrod operations $P^{n}_{k}$ for $n\geq 0$. Let $$i_{L}, i_{R}: H\mathbb{F}^{S}_{p} \to H\mathbb{F}^{S}_{p} \wedge H\mathbb{F}^{S}_{p}$$ denote the left and right $H\mathbb{F}_{p}^{S}$-module maps respectively for $S=D$ (use $\widehat{H}\mathbb{F}^{D}_{p}$), $k$, or $K$. Motivated by the corresponding duality in characteristic $0$, we want to define operations $P^{n}_{k}\in H\mathbb{F}^{k \, *,*}_{p} H\mathbb{F}^{k}_{p}$ for $n \geq 0$ by taking operations dual to the $\xi^{n}_{1}$. 

\begin{definition} 
Let $\alpha$ be a sequence. Define $P^{\alpha}_{k}\in H\mathbb{F}^{k \, *,*}_{p} H\mathbb{F}^{k}_{p}$ by $P^{\alpha}_{k}\coloneqq i^{*}\eta_{\alpha} \circ i_{R}$. For $n\geq 0$, we let $P^{n}_{k}=P^{(0,n,0,\ldots)}_{k}$. Let $\beta_{k}=P^{(1,0,\ldots)}_{k}$. 
\end{definition}

\indent There are corresponding operations $P^{\alpha}_{K}$  in characteristic $0$ defined from \ref{eq:char0dual} by 
\[
\begin{tikzcd}
H\mathbb{F}^{K}_{p} \arrow[r, "i_{R}"] & H\mathbb{F}^{K}_{p}\wedge H\mathbb{F}^{K}_{p} \arrow[r, "proj."] & \Sigma^{p_{\alpha},q_{\alpha}}H\mathbb{F}^{K}_{p}.
\end{tikzcd}
\]

\begin{definition}
 To define a homomorphism $\Phi:H\mathbb{F}^{K\,  *,*}_{p} H\mathbb{F}^{K}_{p} \to H\mathbb{F}^{k\,  *,*}_{p} H\mathbb{F}^{k}_{p}$ of graded additive groups, let $f: H\mathbb{F}^{K}_{p} \to \Sigma ^{k,l}H\mathbb{F}^{K}_{p}$ be given. Define $\Phi(f): H\mathbb{F}^{k}_{p} \to \Sigma ^{k,l}H\mathbb{F}^{k}_{p}$ by $\Phi(f)=\pi \circ i^{*}j_{*}(f) \circ i^{*}\eta$.
\begin{equation} \label{eq:phi def}
\begin{tikzcd}
H\mathbb{F}^{k}_{p} \arrow[r, "i^{*}\eta"] 
          & i^{*}j_{*}H\mathbb{F}^{K}_{p} \arrow[r, "i^{*}j_{*}(f)"] 
               & \Sigma^{k,l}i^{*}j_{*}H\mathbb{F}^{K}_{p} \arrow[r, "\pi"] 
                     & \Sigma^{k,l}H\mathbb{F}^{k}_{p} 
\end{tikzcd}
\end{equation}
\end{definition}

\indent From the definition of $\Phi$, it is clear that $\Phi(id.)=id.$. The following lemma will be important for proving that the operations $P^{n}_{k}$ restricted to mod $p$ Chow groups satisfy the Adem relations and Cartan formula.

\begin{lemma}  \label{lemmafornexthm}

Let $X \in \textup{Sm}_{k}$ and let $f: \Sigma^{\infty}_{+}X \to \Sigma^{2m,m}H\mathbb{F}_{p}^{k}$ be given.

\begin{enumerate}
\item Let $\alpha_{0}$ be a sequence. Consider the morphism $$g_{\alpha_{0}}:H\mathbb{F}^{k}_{p} \to \Sigma^{p_{\alpha_{0}}-1,q_{\alpha_{0}}-1}H\mathbb{F}^{k}_{p}$$  given by the following composite. 
\[
\begin{tikzcd}
H\mathbb{F}^{k}_{p} \arrow[r, "i^{*}\eta"]
& i^{*}j_{*}H\mathbb{F}^{K}_{p} \arrow[r, "i^{*}j_{*}(P^{\alpha_{0}}_{K})"]
& i^{*}j_{*}\Sigma^{p_{\alpha_{0}},q_{\alpha_{0}}}H\mathbb{F}^{K}_{p} \arrow[r, "\pi_{0}"]
& \Sigma^{p_{\alpha_{0}}-1,q_{\alpha_{0}}-1}H\mathbb{F}^{k}_{p}.
\end{tikzcd}
\]
Then $\Sigma^{2m,m}g_{\alpha_{0}} \circ f=0$.

\item The composite 

\[
\begin{tikzcd}
\Sigma^{\infty}_{+}X \arrow[r, "f"] & \Sigma^{2m,m}H\mathbb{F}_{p}^{k} \arrow[r, "i_{R}"] & \Sigma^{2m,m}H\mathbb{F}_{p}^{k} \wedge H\mathbb{F}_{p}^{k} \arrow[r, "i^{*}\eta"] & i^{*}j_{*}(\Sigma^{2m,m}H\mathbb{F}_{p}^{K} \wedge H\mathbb{F}_{p}^{K}) \arrow[d, "i^{*}j_{*}\Psi_{K}^{-1}"] \\
& & & \bigoplus\limits _{\alpha} \Sigma ^{p_{\alpha}+2m,q_{\alpha}+m}i^{*}j_{*}H\mathbb{F}^{K}_{p} \arrow[d,"\oplus \pi_{0}"] \\
& & & \bigoplus\limits _{\alpha} \Sigma ^{2m+p_{\alpha}-1,m+q_{\alpha}-1}H\mathbb{F}^{k}_{p}
\end{tikzcd}
\]
is equal to $0$.

\end{enumerate}
\end{lemma}

\begin{proof}
\indent Note that for any sequence $\alpha$ of bidegree $(p_{\alpha},q_{\alpha})$, $p_{\alpha} \geq 2q_{\alpha}$ which implies that $p_{\alpha}-1> 2(q_{\alpha}-1)$. For $(1)$ and $(2)$, Theorem \ref{voesusfri} implies that $$\textup{Hom}_{SH(k)}(\Sigma^{\infty}_{+}X, \Sigma^{2m+p_{\alpha}-1, m+q_{\alpha}-1}H\mathbb{F}_{p}^{k})=H^{2m+p_{\alpha}-1, m+q_{\alpha}-1}(X, \mathbb{F}_{p})=0$$ for any sequence $\alpha$.

\end{proof}

\begin{theorem} \phantomsection   \label{T:phiproperties}

\begin{enumerate} 
\item We have $\Phi(H^{*,*}(K, \mathbb{F}_{p}))\subset H^{*,*}(k, \mathbb{F}_{p}).$
\item Let $\alpha$ be a sequence. Then $\Phi(P^{\alpha}_{K})=P^{\alpha}_{k}$.
In particular, for the Bockstein $\beta_{K}$ and reduced power operations $P^{n}_{K}$ constructed by Voevodsky in characteristic $0$, $\Phi(P^{n}_{K})=P^{n}_{k}$ for $n \geq 0$ and $\Phi(\beta_{K})=\beta_{k}$. Also, $P^{0}_{k}$ is the identity since $P^{0}_{K}$ is the identity.
\item  Let $X \in \textup{Sm}_{k}$ and let $f: \Sigma^{\infty}_{+}X \to \Sigma^{2m,m}H\mathbb{F}_{p}^{k}$ be given. Let $\alpha$ be a sequence and let $h:H\mathbb{F}_{p}^{K} \to \Sigma^{i,j}H\mathbb{F}_{p}^{K}$ be given. Then $$\Phi(h\circ P^{\alpha}_{K})(f)=\Phi(h)(P^{\alpha}_{k}(f)).$$
\end{enumerate}

\end{theorem}

\begin{proof}
\indent We first prove $(1)$. Let $a\in H^{*,*}(K, \mathbb{F}_{p})$. The element $a$ corresponds to a morphism $f_{a}: H\mathbb{F}^{K}_{p} \to\Sigma^{m,n}H\mathbb{F}^{K}_{p}$ in $D(H\mathbb{F}^{K}_{p})$. The functors $i^{*}$, $j_{*}$ restrict to functors $i^{*}:D(\widehat{H}\mathbb{F}^{D}_{p})\to D(H\mathbb{F}^{k}_{p})$ and $j_{*}:D(H\mathbb{F}^{K}_{p})\to D(\widehat{H}\mathbb{F}^{D}_{p})$. Hence, $i^{*}j_{*}(f_{a})$ is a morphism in $D(H\mathbb{F}^{k}_{p})$. From the definition of $\Phi$, it follows that $\Phi(f_{a})$ is a morphism in $D(H\mathbb{F}^{k}_{p}).$ Thus, $\Phi(a) \coloneqq \Phi(f_{a})\in H^{*,*}(k, \mathbb{F}_{p}).$

\indent We now prove $(2)$. Let $\alpha$ be a sequence. Applying the natural transformation $i^{*} \to i^{*}j_{*}j^{*}$ to $i_{R}: \widehat{H}\mathbb{F}^{D}_{p} \to \widehat{H}\mathbb{F}^{D}_{p} \wedge \widehat{H}\mathbb{F}^{D}_{p}$, we obtain the following commuting square in $SH(k)$. 
\[
\begin{tikzcd}
H\mathbb{F}^{k}_{p} \arrow[r, "i_{R}"] \arrow[d, "i^{*}\eta"]
          & H\mathbb{F}^{k}_{p}\wedge H\mathbb{F}^{k}_{p} \arrow[d, "i^{*}\eta"] \\
i^{*}j_{*}H\mathbb{F}^{K}_{p} \arrow[r, "i^{*}j_{*}(i_{R})"] 
          &  i^{*}j_{*}(H\mathbb{F}^{K}_{p}\wedge H\mathbb{F}^{K}_{p})
\end{tikzcd}
\]
From the definition of $i^{*}\eta_{\alpha}$ \ref{eq:retractexplicit r}, the following diagram commutes.
\[
\begin{tikzcd}
H\mathbb{F}^{k}_{p}\wedge H\mathbb{F}^{k}_{p} \arrow[r, "i^{*}\eta_{\alpha}"] \arrow[d, "i^{*}\eta"]
          & \Sigma^{p_{\alpha},q_{\alpha}}H\mathbb{F}^{k}_{p} \arrow[r, "id."] 
          &  \Sigma^{p_{\alpha},q_{\alpha}}H\mathbb{F}^{k}_{p} \arrow[d, "id."] \\
i^{*}j_{*}(H\mathbb{F}^{K}_{p}\wedge H\mathbb{F}^{K}_{p}) \arrow[r, "proj."] 
          &  i^{*}j_{*}\Sigma^{p_{\alpha},q_{\alpha}}H\mathbb{F}^{K}_{p} \arrow[r, "\pi"]
          &  \Sigma^{p_{\alpha},q_{\alpha}}H\mathbb{F}^{k}_{p}
\end{tikzcd}
\]

\indent Putting these 2 diagrams together, we get the following commuting diagram.

\begin{equation} \label{eq:use for next ref}
\begin{tikzcd}
H\mathbb{F}^{k}_{p} \arrow[r, "i_{R}"] \arrow[d, "i^{*}\eta"]
          & H\mathbb{F}^{k}_{p}\wedge H\mathbb{F}^{k}_{p} \arrow[r, "i^{*}\eta_{\alpha}"] \arrow[d, "i^{*}\eta"] 
          & \Sigma^{p_{\alpha},q_{\alpha}}H\mathbb{F}^{k}_{p} \arrow[r, "id."] 
          &  \Sigma^{p_{\alpha},q_{\alpha}}H\mathbb{F}^{k}_{p} \arrow[d, "id."] \\
i^{*}j_{*}H\mathbb{F}^{K}_{p} \arrow[r, "i^{*}j_{*}(i_{R})"] 
          &  i^{*}j_{*}(H\mathbb{F}^{K}_{p}\wedge H\mathbb{F}^{K}_{p}) \arrow[r, "proj."]
          &  i^{*}j_{*}\Sigma^{p_{\alpha},q_{\alpha}}H\mathbb{F}^{K}_{p} \arrow[r, "\pi"]
          &  \Sigma^{p_{\alpha},q_{\alpha}}H\mathbb{F}^{k}_{p}
\end{tikzcd}
\end{equation}                  
The top row of this diagram gives $P^{\alpha}_{k}$ while the composite starting at $H\mathbb{F}^{k}_{p}$ in the top left and continuing along the bottom row gives $\Phi(P^{\alpha}_{K})$. Hence, $\Phi(P^{\alpha}_{K})=P^{\alpha}_{k}$.

\indent Now, we prove $(3)$. Consider the following diagram.
\begin{equation}\label{eq:phi commut diag}
\begin{tikzcd}            
\Sigma^{\infty}_{+}X \arrow[d, "f"] \\ 
 \Sigma^{2m,m}H\mathbb{F}^{k}_{p} \arrow[r, "P^{\alpha}_{k}"] \arrow[d, "i^{*}\eta"]
&\Sigma^{2m+p_{\alpha},m+q_{\alpha}}H\mathbb{F}^{k}_{p} \arrow[r, "\Phi(h)"] \arrow[d, "i^{*}\eta"]
& \Sigma^{i+2m+p_{\alpha},j+m+q_{\alpha}}H\mathbb{F}^{k}_{p}  \arrow[d, "i^{*}\eta"] \\
 i^{*}j_{*}\Sigma^{2m,m}H\mathbb{F}^{K}_{p} \arrow[r, "i^{*}j_{*}P^{\alpha}_{K}"]             
& i^{*}j_{*}\Sigma^{2m+p_{\alpha},m+q_{\alpha}}H\mathbb{F}^{K}_{p} \arrow[r, "i^{*}j_{*}h"]
& i^{*}j_{*}\Sigma^{i+2m+p_{\alpha},j+m+q_{\alpha}}H\mathbb{F}^{K}_{p} \arrow[d, "\pi"]\\
& & \Sigma^{i+2m+p_{\alpha},j+m+q_{\alpha}}H\mathbb{F}^{k}_{p}.
\end{tikzcd}
\end{equation}

As $\Phi(P^{\alpha}_{K})=P^{\alpha}_{k},$ Lemma \ref{lemmafornexthm} implies that the composite $$i^{*}\eta \circ P^{\alpha}_{k} \circ f:\Sigma^{\infty}_{+}X \to i^{*}j_{*}\Sigma^{2m+p_{\alpha},m+q_{\alpha}}H\mathbb{F}^{K}_{p}$$ in diagram \ref{eq:phi commut diag} is equal to $$i^{*}j_{*}P^{\alpha}_{K} \circ i^{*}\eta \circ f.$$ Equivalently, $$\pi_{0} \circ i^{*}j_{*}P^{\alpha}_{K} \circ i^{*}\eta \circ f=0:\Sigma^{\infty}_{+}X \to \Sigma^{2m+p_{\alpha}-1,m+q_{\alpha}-1}H\mathbb{F}^{k}_{p} .$$ 

 Thus, from diagram \ref{eq:phi commut diag}, $$\Phi(h)(P^{\alpha}_{k}(f))=\pi \circ i^{*}\eta \circ \Phi(h) \circ P^{\alpha}_{k} \circ f=\pi \circ i^{*}j_{*}(h)\circ i^{*}j_{*}(P^{\alpha}_{K}) \circ i^{*}\eta \circ f= \Phi(h\circ P^{\alpha}_{K})(f)$$ as desired.

\end{proof}
 
\indent We next prove that the operations $P^{n}_{k}$ commute with base change of the field $k$ on mod $p$ Chow groups. For a morphism of fields $f:\textup{Spec}(F_{1}) \to \textup{Spec}(F_{2})$, the pullback functor $f^{*}:SH(F_{2}) \to SH(F_{1})$ induces a homomorphism $H\mathbb{F}^{F_{2} \, *,*}_{p} H\mathbb{F}^{F_{2}}_{p} \to H\mathbb{F}^{F_{1} \, *,*}_{p} H\mathbb{F}^{F_{1}}_{p}$. For $\textup{char}(F_{2}) \neq p$, $f^{*}(P^{n}_{F_{2}})=P^{n}_{F_{1}}$ since the dual Steenrod algebra has the expected form in this case \cite[Theorem 1.1]{HKO}. However, for our situation where the base field is of characteristic $p$, we do not yet know the full structure of the dual Steenrod algebra.

\indent Let $f_{1}:\textup{Spec}(k) \to \textup{Spec}(\mathbb{F}_{p})$ be the structure map. In the following commuting diagram, $f_{2}$, $f_{3}$, $i_{0}$, and $j_{0}$ are maps compatible with $f_{1}$.
\[
\begin{tikzcd}
\textup{Spec}(k) \arrow[r, "f_{1}"] \arrow[d, "i"] & \textup{Spec}(\mathbb{F}_{p}) \arrow[d, "i_{0}"] \\
\textup{Spec}(D) \arrow[r, "f_{2}"] & \textup{Spec}(\mathbb{Z}_{p}) \\
\textup{Spec}(K) \arrow[r, "f_{3}"] \arrow[u, "j"] & \textup{Spec}(\mathbb{Q}_{p}) \arrow[u, "j_{0}"]
\end{tikzcd}
\]

\begin{proposition} \label{P:pullback}
Let $X \in \textup{Sm}_{k}$ and let $g: \Sigma^{\infty}_{+}X \to \Sigma^{2m,m}H\mathbb{F}_{p}^{k}$ be given. Then $P^{n}_{k}(g)=f_{1}^{*}(P^{n}_{\mathbb{F}_{p}})(g)$ for all $n \geq 0$.
\end{proposition}

\begin{proof}
Let $\eta_{0}:1 \to j_{0 \,*}j^{*}_{0}$ denote the unit map. Let $f_{2}^{*}\widehat{H}\mathbb{F}_{p}^{\mathbb{Z}_{p}} \to f_{2}^{*}j_{0 \,*}H\mathbb{F}_{p}^{\mathbb{Q}_{p}}$ be the map $f_{2}^{*}\eta_{0}$ induced by the isomorphism $j_{0}^{*}\widehat{H}\mathbb{F}_{p}^{\mathbb{Z}_{p}} \to H\mathbb{F}_{p}^{\mathbb{Q}_{p}}$. The exchange transformation $f_{2}^{*}j_{0 \,*}\to j_{*}f_{3}^{*}$ induces a morphism $f_{2}^{*}j_{0 \,*}H\mathbb{F}_{p}^{\mathbb{Q}_{p}} \to j_{*}f^{*}_{3}H\mathbb{F}_{p}^{\mathbb{Q}_{p}}$. Let $f_{2}^{*}\widehat{H}\mathbb{F}_{p}^{\mathbb{Z}_{p}} \to j_{*}f^{*}_{3}H\mathbb{F}_{p}^{\mathbb{Q}_{p}}$ be the map $\eta f^{*}_{2}$ induced by the isomorphism $$j^{*}f^{*}_{2} \widehat{H}\mathbb{F}_{p}^{\mathbb{Z}_{p}}\cong f^{*}_{3}j^{*}_{0}\widehat{H}\mathbb{F}_{p}^{\mathbb{Z}_{p}} \to f_{3}^{*}H\mathbb{F}_{p}^{\mathbb{Q}_{p}}.$$ Putting these maps together, we get the following square which commutes by adjunction.

\begin{equation} \label{diagpullback2}
\begin{tikzcd}
f_{2}^{*}\widehat{H}\mathbb{F}_{p}^{\mathbb{Z}_{p}} \arrow[r, "f_{2}^{*}\eta_{0}"] \arrow[d, "id."] & f_{2}^{*}j_{0 \,*}H\mathbb{F}_{p}^{\mathbb{Q}_{p}} \arrow[d, ""] \\
f_{2}^{*}\widehat{H}\mathbb{F}_{p}^{\mathbb{Z}_{p}} \arrow[r, "\eta f^{*}_{2}"] & j_{*}f^{*}_{3}H\mathbb{F}_{p}^{\mathbb{Q}_{p}}
\end{tikzcd}
\end{equation}
 \indent Applying the exchange transformation $f_{2}^{*}j_{0 \,*}\to j_{*}f_{3}^{*}$ to $P^{n}_{\mathbb{Q}_{p}}$, we get the following commuting square.
 
\[
\begin{tikzcd}
f_{2}^{*}j_{0 \,*}H\mathbb{F}_{p}^{\mathbb{Q}_{p}} \arrow[r,"f_{2}^{*}j_{0 \,*}P^{n}_{\mathbb{Q}_{p}}"] \arrow[d, ""] & f_{2}^{*}j_{0 \,*}\Sigma^{2n(p-1),n(p-1)}H\mathbb{F}_{p}^{\mathbb{Q}_{p}} \arrow[d,""] \\
j_{*}H\mathbb{F}_{p}^{K} \arrow[r, "j_{*}P^{n}_{K}"] & j_{*}\Sigma^{2n(p-1),n(p-1)}H\mathbb{F}_{p}^{K} 
\end{tikzcd}
\]
Applying $i^{*}$ (and the connection isomorphism $i^{*}f_{2}^{*} \cong f^{*}_{1}i^{*}_{0}$) to these two squares and combining with $g: \Sigma^{\infty}_{+}X \to \Sigma^{2m,m}H\mathbb{F}_{p}^{k}$ , we obtain the following commuting diagram.

\begin{equation} \label{diagpullback3}
\begin{tikzcd}
\Sigma^{\infty}_{+}X \arrow[r, "g"] \arrow[d, "id."] & \Sigma^{2m,m}H\mathbb{F}_{p}^{k} \arrow[r,"f_{1}^{*}i_{0}^{*}\eta_{0}"] \arrow[d, "id."] & f_{1}^{*}i_{0}^{*} j_{0 \,*}\Sigma^{2m,m}H\mathbb{F}_{p}^{\mathbb{Q}_{p}} \arrow[r, "f_{1}^{*}i_{0}^{*} j_{0 \,*}P^{n}_{\mathbb{Q}_{p}}"]\arrow[d, ""]
& f_{1}^{*}i_{0}^{*} j_{0 \,*}\Sigma^{2(m+n(p-1)),m+n(p-1)}H\mathbb{F}_{p}^{\mathbb{Q}_{p}} \arrow[d, ""]\\
\Sigma^{\infty}_{+}X \arrow[r, "g"]  & \Sigma^{2m,m}H\mathbb{F}_{p}^{k} \arrow[r, "i^{*}\eta"]& i^{*}j_{*}\Sigma^{2m,m}H\mathbb{F}_{p}^{K} \arrow[r, "i^{*}j_{*}P^{n}_{K}"] & i^{*}j_{*}\Sigma^{2(m+n(p-1)),m+n(p-1)}H\mathbb{F}_{p}^{K}
\end{tikzcd}
\end{equation}

Let $\pi':i_{0}^{*} j_{0 \,*}H\mathbb{F}_{p}^{\mathbb{Q}_{p}} \to H\mathbb{F}_{p}^{\mathbb{F}_{p}}$ and $\pi_{0}':i_{0}^{*} j_{0 \,*}H\mathbb{F}_{p}^{\mathbb{Q}_{p}} \to \Sigma^{-1,-1}H\mathbb{F}_{p}^{\mathbb{F}_{p}}$ be projection morphisms induced by the isomorphism $i_{0}^{*} j_{0 \,*}H\mathbb{F}_{p}^{\mathbb{Q}_{p}} \cong H\mathbb{F}_{p}^{\mathbb{F}_{p}} \oplus \Sigma^{-1,-1}H\mathbb{F}_{p}^{\mathbb{F}_{p}}$ of Theorem \ref{T:defofpi0andpi}. From Theorem \ref{voesusfri}, the two composites $\Sigma^{\infty}_{+}X \to \Sigma^{2(m+n(p-1))-1,m+n(p-1)-1}H\mathbb{F}_{p}^{k}$ given by the following diagram are equal to $0$.

\begin{equation} \label{diagpullback4}
\begin{tikzcd}
& & &  \Sigma^{2(m+n(p-1))-1,m+n(p-1)-1}H\mathbb{F}_{p}^{k}\\
\Sigma^{\infty}_{+}X \arrow[r, "g"] \arrow[d, "id."] & \Sigma^{2m,m}H\mathbb{F}_{p}^{k} \arrow[r,"f_{1}^{*}i_{0}^{*}\eta_{0}"] \arrow[d, "id."] & f_{1}^{*}i_{0}^{*} j_{0 \,*}\Sigma^{2m,m}H\mathbb{F}_{p}^{\mathbb{Q}_{p}} \arrow[r, "f_{1}^{*}i_{0}^{*} j_{0 \,*}P^{n}_{\mathbb{Q}_{p}}"]\arrow[d, ""]
& f_{1}^{*}i_{0}^{*} j_{0 \,*}\Sigma^{2(m+n(p-1)),m+n(p-1)}H\mathbb{F}_{p}^{\mathbb{Q}_{p}} \arrow[d, ""] \arrow[u, "f^{*}_{1}\pi_{0}'"]\\
\Sigma^{\infty}_{+}X \arrow[r, "g"]  & \Sigma^{2m,m}H\mathbb{F}_{p}^{k} \arrow[r, "i^{*}\eta"]& i^{*}j_{*}\Sigma^{2m,m}H\mathbb{F}_{p}^{K} \arrow[r, "i^{*}j_{*}P^{n}_{K}"] & i^{*}j_{*}\Sigma^{2(m+n(p-1)),m+n(p-1)}H\mathbb{F}_{p}^{K} \arrow[d, "\pi_{0}"] \\
& & &  \Sigma^{2(m+n(p-1))-1,m+n(p-1)-1}H\mathbb{F}_{p}^{k}
\end{tikzcd}
\end{equation}

Consider the following diagram.

\begin{equation} \label{diagpullback}
\begin{tikzcd}
& & &  \Sigma^{2(m+n(p-1)),m+n(p-1)}H\mathbb{F}_{p}^{k}\\
\Sigma^{\infty}_{+}X \arrow[r, "g"] \arrow[d, "id."] & \Sigma^{2m,m}H\mathbb{F}_{p}^{k} \arrow[r,"f_{1}^{*}i_{0}^{*}\eta_{0}"] \arrow[d, "id."] & f_{1}^{*}i_{0}^{*} j_{0 \,*}\Sigma^{2m,m}H\mathbb{F}_{p}^{\mathbb{Q}_{p}} \arrow[r, "f_{1}^{*}i_{0}^{*} j_{0 \,*}P^{n}_{\mathbb{Q}_{p}}"]\arrow[d, ""]
& f_{1}^{*}i_{0}^{*} j_{0 \,*}\Sigma^{2(m+n(p-1)),m+n(p-1)}H\mathbb{F}_{p}^{\mathbb{Q}_{p}} \arrow[d, ""] \arrow[u, "f^{*}_{1}\pi'"]\\
\Sigma^{\infty}_{+}X \arrow[r, "g"]  & \Sigma^{2m,m}H\mathbb{F}_{p}^{k} \arrow[r, "i^{*}\eta"]& i^{*}j_{*}\Sigma^{2m,m}H\mathbb{F}_{p}^{K} \arrow[r, "i^{*}j_{*}P^{n}_{K}"] & i^{*}j_{*}\Sigma^{2(m+n(p-1)),m+n(p-1)}H\mathbb{F}_{p}^{K} \arrow[d, "\pi"] \\
& & &  \Sigma^{2(m+n(p-1)),m+n(p-1)}H\mathbb{F}_{p}^{k}
\end{tikzcd}
\end{equation}
From Theorem \ref{T:phiproperties}, the composite $\Sigma^{\infty}_{+}X \to \Sigma^{2(m+n(p-1)),m+n(p-1)}H\mathbb{F}_{p}^{k}$ given by the upper half of diagram \ref{diagpullback} is equal to $f_{1}^{*}(P^{n}_{\mathbb{F}_{p}})(g)$ and the composite $\Sigma^{\infty}_{+}X \to \Sigma^{2(m+n(p-1)),m+n(p-1)}H\mathbb{F}_{p}^{k}$ given by the lower half of diagram \ref{diagpullback} is equal to $P^{n}_{k}(g)$. As diagram \ref{diagpullback3} commutes and the $2$ composite morphisms from diagram \ref{diagpullback4} are $0$, we then obtain that $f_{1}^{*}(P^{n}_{\mathbb{F}_{p}})(g)=P^{n}_{k}(g)$.
\end{proof}

\indent We can now prove that the Steenrod operations $P^{n}_{k}$ commute with base change on mod $p$ Chow groups. Let $f:\textup{Spec}(k_{1}) \to \textup{Spec}(k_{2})$ be given where $k_{1}, k_{2}$ are fields of characteristic $p$. Let $h:\textup{Spec}(k_{2}) \to \textup{Spec}(\mathbb{F}_{p})$ be the structure map.
\begin{corollary} \label{corbasechange}
Let $X \in \textup{Sm}_{k_{2}}$. Let $n \geq 0$. The following square commutes.

\[
\begin{tikzcd}
CH^{*}(X)/p \arrow[r, "P^{n}_{k_{2}}"] \arrow[d, "f^{*}"] & 
CH^{*}(X)/p \arrow[d, "f^{*}"] \\
CH^{*}(X_{k_{1}})/p \arrow[r, "P^{n}_{k_{1}}"] & CH^{*}(X_{k_{1}})/p
\end{tikzcd}
\]
\end{corollary}

\begin{proof}
From Proposition \ref{P:pullback}, $h^{*}P^{n}_{\mathbb{F}_{p}}$ agrees with $P^{n}_{k_{2}}$ on $CH^{*}(X)/p$ and $f^{*}h^{*}P^{n}_{\mathbb{F}_{p}}$ agrees with $P^{n}_{k_{1}}$ on $CH^{*}(X_{k_{1}})/p $. Let $g:\Sigma^{\infty}_{+}X \to \Sigma^{2m,m}H\mathbb{F}_{p}^{k_{2}}$ be given. Then $$f^{*}(P^{n}_{k_{2}}(g))=f^{*}(h^{*}P^{n}_{\mathbb{F}_{p}}(g))=f^{*}h^{*}(P^{n}_{\mathbb{F}_{p}})(f^{*}g)=P^{n}_{k_{1}}(f^{*}g)$$ as required.
\end{proof}
 
 \begin{proposition}
The morphism $\beta_{k}=P^{(1,0,\ldots)}_{k}$ defined above is equal to the Bockstein homomorphism $\beta$ on mod $p$ motivic cohomology.
\end{proposition}
\begin{proof}
We let $\beta$ denote the Bockstein homomorphism on mod $p$ motivic cohomology over any base scheme. The Bockstein homomorphism $\beta$ in characteristic $0$ is known to be dual to $\tau_{0}$. Hence, $\beta=P^{(1,0,\ldots)}_{K}=\beta_{K}$. Applying the natural transformation $i^{*} \to i^{*}j_{*}j^{*}$ to the diagram
\[
\begin{tikzcd}
\widehat{H}\mathbb{Z}^{D} \arrow[r, "\cdot p"]
& \widehat{H}\mathbb{Z}^{D} \arrow[r, ""]
&(\widehat{H}\mathbb{Z}^{D})/p \arrow[r, ""]  \arrow[rr, bend left, "\beta"]
& \Sigma^{1,0}\widehat{H}\mathbb{Z}^{D} \arrow[r, "proj."]
& \Sigma^{1,0}\widehat{H}\mathbb{Z}^{D}/p
\end{tikzcd}
\]
in $SH(D)$, we get the following commuting diagram in $SH(k)$. 
\begin{equation} \label{eq:bocktri}
\begin{tikzcd}
H\mathbb{Z}^{k} \arrow[r, "\cdot p"] \arrow[d, "i^{*}\eta"]
& H\mathbb{Z}^{k} \arrow[r, "proj."] \arrow[d, "i^{*}\eta"]
& H\mathbb{F}^{k}_{p} \arrow[r, ""]  \arrow[rr, bend left, "\beta"] \arrow[d, "i^{*}\eta"]
& \Sigma^{1,0}H\mathbb{Z}^{k} \arrow[r, "proj."] \arrow[d, "i^{*}\eta"]
& \Sigma^{1,0}H\mathbb{F}^{k}_{p} \arrow[d, "i^{*}\eta"] \\
i^{*}j_{*}H\mathbb{Z}^{K} \arrow[r, "\cdot p"] 
& i^{*}j_{*}H\mathbb{Z}^{K} \arrow[r, "proj."] 
& i^{*}j_{*}H\mathbb{F}^{K}_{p} \arrow[r, ""]  \arrow[rr, bend right, "i^{*}j_{*}\beta_{K}"] 
& \Sigma^{1,0}i^{*}j_{*}H\mathbb{Z}^{K} \arrow[r, "proj."] 
& \Sigma^{1,0}i^{*}j_{*}H\mathbb{F}^{K}_{p} \arrow[d, "\pi"] \\
 & & & &      \Sigma^{1,0}H\mathbb{F}^{k}_{p}
\end{tikzcd}
\end{equation}
From Theorem \ref{T:phiproperties}, $\Phi(\beta_{K})=\beta_{k}$. The composite in diagram \ref{eq:bocktri} that starts at $H\mathbb{F}^{k}_{p}$ in the top row and goes immediately down to $\Sigma^{1,0}H\mathbb{F}^{k}_{p}$ is equal to $\Phi(\beta_{K})$. As the diagram commutes and $\pi \circ i^{*}\eta=id.$, it follows that $\Phi(\beta_{K})=\beta=\beta_{k}$.
\end{proof}

\section{Adem relations}

\indent In this section, we use the map $\Phi:H\mathbb{F}^{K\,  *,*}_{p} H\mathbb{F}^{K}_{p} \to H\mathbb{F}^{k\,  *,*}_{p} H\mathbb{F}^{k}_{p}$ \ref{eq:phi def} and Theorem \ref{T:phiproperties} to show that the operations $P^{n}_{k}$ for $n \geq 0$ satisfy the expected Adem relations when restricted to mod $p$ Chow groups. The proof uses the corresponding Adem relations in characteristic $0$ which can be found in \cite[Th\'eor\`eme 4.5.1]{Rio} for $p=2$ and \cite[Th\'eor\`eme 4.5.2 ]{Rio} for odd $p$. First, we state the Adem relations for $p=2$ over the base $K$ of characteristic $0$. Let $\tau \in H^{0,1}(K, \mathbb{F}_{2})$ denote the class of $-1 \in \mu_{2}(K)$ and let $\rho \in H^{1,1}(K, \mathbb{F}_{2})$ denote the class of $-1 \in K^{*}/K^{* \, 2}$. Set $\textrm{Sq}^{2n}_{k}\coloneqq P^{n}_{k}$ and $\textrm{Sq}^{2n+1}_{k}=\beta_{k}\textrm{Sq}^{2n}_{k}$  for $n \geq0$.

\begin{theorem} \label{T:char0adem}
Let $a,b \in \mathbb{N}$ with $a<2b$. 
\begin{enumerate}
\item $$\mathrm{Sq}^{a}_{K}\mathrm{Sq}^{b}_{K}=  
      \sum_{j=0}^{\lfloor \frac{a}{2} \rfloor}\binom{b-1-j}{a-2j}\mathrm{Sq}^{a+b-j}_{K}\mathrm{Sq}^{j}_{K} +\sum_{\substack{j=1 \\ j\, odd}}^{\lfloor \frac{a}{2} \rfloor}\rho \binom{b-1-j}{a-2j} \mathrm{Sq}^{a+b-j-1}_{K}\mathrm{Sq}^{j}_{K}$$ if $a$ is even and $b$ is odd.\\          

\item  $$\mathrm{Sq}^{a}_{K}\mathrm{Sq}^{b}_{K}=  
      \sum_{\substack{j=0 \\ j\, odd}}^{\lfloor \frac{a}{2} \rfloor}\binom{b-1-j}{a-2j}\mathrm{Sq}^{a+b-j}_{K}\mathrm{Sq}^{j}_{K}$$ if $a$ and $b$ are odd.\\
      
\item $$\mathrm{Sq}^{a}_{K}\mathrm{Sq}^{b}_{K}=  
      \sum_{\substack{j=0 \\ }}^{\lfloor \frac{a}{2} \rfloor}\tau^{j \, \textup{mod}\, 2}\binom{b-1-j}{a-2j}\mathrm{Sq}^{a+b-j}_{K}\mathrm{Sq}^{j}_{K}$$ if $a$ and $b$ are even. \\
      
\item $$\mathrm{Sq}^{a}_{K}\mathrm{Sq}^{b}_{K}=  
      \sum_{\substack{j=0 \\ j\, even}}^{\lfloor \frac{a}{2} \rfloor}\binom{b-1-j}{a-2j}\mathrm{Sq}^{a+b-j}_{K}\mathrm{Sq}^{j}_{K} +\sum_{\substack{j=1 \\ j\, odd}}^{\lfloor \frac{a}{2} \rfloor}\rho \binom{b-1-j}{a-1-2j}\mathrm{Sq}^{a+b-j-1}_{K}\mathrm{Sq}^{j}_{K}$$ if $a$ is odd and $b$ is even.
    
   \end{enumerate}

\end{theorem}

\indent Next, we state the characteristic $0$ Adem relations for $p$ odd.

\begin{theorem}
\begin{enumerate}
\item Let $a,b \in \mathbb{N}$ with $a<pb$. Then $$P^{a}_{K}P^{b}_{K}=\sum^{\lfloor \frac{a}{p} \rfloor}_{j=0}(-1)^{a+j}\binom{(p-1)(b-j)-1}{a-pj}P^{a+b-j}_{K}P^{j}_{K}.$$

\item Let $a,b \in \mathbb{N}$ with $a\leq pb$. Then $$P^{a}_{K}\beta_{K}P^{b}_{K}=\sum^{\lfloor \frac{a}{p} \rfloor}_{j=0}(-1)^{a+j}\binom{(p-1)(b-j)-1}{a-pj}\beta_{K}P^{a+b-j}_{K}P^{j}_{K}+$$ 

$$\sum^{\lfloor \frac{a-1}{p} \rfloor}_{j=0}(-1)^{a+j+1}\binom{(p-1)(b-j)-1}{a-pj-1}P^{a+b-j}_{K}\beta_{K}P^{j}_{K}.$$
\end{enumerate}
\end{theorem}

\indent We can now prove the Adem relations for the operations $P^{n}_{k}$ restricted to mod $p$ Chow groups.

\begin{theorem}
Let $X \in \textup{Sm}_{k}$ and let $x \in H^{2m,m}(X,\mathbb{F}_{p})=CH^{m}(X)/p$ for some $m \geq 0$. Let $a,b \in \mathbb{N}$ such that $a<pb$. Then $$P^{a}_{k}(P^{b}_{k}(x))=\sum^{\lfloor \frac{a}{p} \rfloor}_{j=0}(-1)^{a+j}\binom{(p-1)(b-j)-1}{a-pj}P^{a+b-j}_{k}(P^{j}_{k}(x)).$$

\end{theorem}

\begin{proof}
From  Theorem \ref{T:phiproperties}, $P^{a}_{k}(P^{b}_{k}(x))=\Phi(P^{a}_{K}P^{b}_{K})(x)$. We then use the Adem relations in characteristic $0$ to rewrite $P^{a}_{K}P^{b}_{K} \in H\mathbb{F}^{K \, *,*}_{p} H\mathbb{F}^{K}_{p}$. Note that the Bockstein $\beta_{k}$ is the $0$ homomorphism on mod $p$ Chow groups. If $p=2$, $\Phi(\textup{Sq}^{n}_{K})(x)= \textup{Sq}^{n}_{k}(x)=0$ whenever $n$ is odd. Thus, applying Theorem \ref{T:phiproperties}, we get $$P^{a}_{k}(P^{b}_{k}(x))=\Phi(P^{a}_{K}P^{b}_{K})(x)=\Phi(\sum^{\lfloor \frac{a}{p} \rfloor}_{j=0}(-1)^{a+j}\binom{(p-1)(b-j)-1}{a-pj}P^{a+b-j}_{K}P^{j}_{K})(x)$$

$$=\sum^{\lfloor \frac{a}{p} \rfloor}_{j=0}(-1)^{a+j}\binom{(p-1)(b-j)-1}{a-pj}P^{a+b-j}_{k}(P^{j}_{k}(x)).$$
\end{proof}

\section{Coaction map for smooth $X$ }

\indent In this section, for $X\in \textup{Sm}_{k}$, we describe a coaction map $$\lambda_{X}:H^{*,*}(X, \mathbb{F}_{p})\to \pi_{-*,-*}(\bigoplus\limits _{\alpha} \Sigma ^{p_{\alpha},q_{\alpha}}H\mathbb{F}^{k}_{p}) \otimes_{\pi_{-*,-*}H\mathbb{F}^{k}_{p}} H^{*,*}(X, \mathbb{F}_{p})$$ such that the actions of the cohomology operations $P^{n}_{k}$ defined in Section \ref{sec:def} on $H^{*,*}(X, \mathbb{F}_{p})$ are determined by $\lambda_{X}$. We show that $\lambda_{X}$ is a ring homomorphism when restricted to mod $p$ Chow groups. This will allow us to prove the Cartan formula in the next section.

\indent  There is a multiplication morphism \begin{equation} \label{eq:defofm} m:(\bigoplus\limits _{\alpha} \Sigma ^{p_{\alpha},q_{\alpha}}H\mathbb{F}^{k}_{p}) \wedge (\bigoplus\limits _{\alpha} \Sigma ^{p_{\alpha},q_{\alpha}}H\mathbb{F}^{k}_{p}) \to \bigoplus\limits _{\alpha} \Sigma ^{p_{\alpha},q_{\alpha}}H\mathbb{F}^{k}_{p}\end{equation} defined as $m=r \circ \mu_{2}^{k} \circ (\Psi_{k} \wedge \Psi_{k})$. The morphism $m$ defines multiplication on $$(\bigoplus\limits _{\alpha} \Sigma ^{p_{\alpha},q_{\alpha}}H\mathbb{F}^{k}_{p})^{*,*}(\Sigma^{\infty}_{+}X)$$ and $$\pi_{*,*}(\bigoplus\limits _{\alpha} \Sigma ^{p_{\alpha},q_{\alpha}}H\mathbb{F}^{k}_{p}).$$ For sequences $\alpha_{1}, \alpha_{2}$,  Proposition \ref{P:tau2=0} allows us to calculate the product $$r_{*}(\omega(\alpha_{1}))r_{*}(\omega(\alpha_{2}))\in \pi_{*,*}(\bigoplus\limits _{\alpha} \Sigma ^{p_{\alpha},q_{\alpha}}H\mathbb{F}^{k}_{p})$$ in terms of another sequence $\alpha_{1}+\alpha_{2}$ by using the relations $\tau_{i}^{2}=0$ for $i\geq0$.

\begin{proposition} \label{P:coactidentitensor}
The natural ring homomorphism $$\pi_{-*,-*}(\bigoplus\limits _{\alpha} \Sigma ^{p_{\alpha},q_{\alpha}}H\mathbb{F}^{k}_{p}) \otimes_{\pi_{-*,-*}H\mathbb{F}^{k}_{p}}H^{*,*}(X, \mathbb{F}_{p}) \to (\bigoplus\limits _{\alpha} \Sigma ^{p_{\alpha},q_{\alpha}}H\mathbb{F}^{k}_{p})^{*,*}(\Sigma^{\infty}_{+}X)$$ is an isomorphism.
\end{proposition}
\begin{proof}
The suspension spectrum $\Sigma^{\infty}_{+}X \in SH(k)$ is compact. Hence, $$\mathrm{Hom}_{SH(k)}(\Sigma^{s,t}\Sigma^{\infty}_{+}X, \bigoplus\limits _{\alpha} \Sigma ^{p_{\alpha},q_{\alpha}}H\mathbb{F}^{k}_{p}) \cong \bigoplus\limits _{\alpha}\mathrm{Hom}_{SH(k)}(\Sigma^{s,t}\Sigma^{\infty}_{+}X,  \Sigma ^{p_{\alpha},q_{\alpha}}H\mathbb{F}^{k}_{p})$$ for all $s,t \in \mathbb{Z}$.
\end{proof} 

\begin{definition} Using the isomorphism $$(\bigoplus\limits _{\alpha} \Sigma ^{p_{\alpha},q_{\alpha}}H\mathbb{F}^{k}_{p})^{*,*}(\Sigma^{\infty}_{+}X) \cong \pi_{-*,-*}(\bigoplus\limits _{\alpha} \Sigma ^{p_{\alpha},q_{\alpha}}H\mathbb{F}^{k}_{p}) \otimes_{\pi_{-*,-*}H\mathbb{F}^{k}_{p}}H^{*,*}(X, \mathbb{F}_{p})$$ from Proposition \ref{P:coactidentitensor} , define an additive homomorphism of graded abelian groups $$\lambda_{X}:H^{*,*}(X, \mathbb{F}_{p})\to \pi_{-*,-*}(\bigoplus\limits _{\alpha} \Sigma ^{p_{\alpha},q_{\alpha}}H\mathbb{F}^{k}_{p}) \otimes_{\pi_{-*,-*}H\mathbb{F}^{k}_{p}} H^{*,*}(X, \mathbb{F}_{p})$$ by the composite
\begin{equation} \label{coacitonmapdef}
\begin{tikzcd}
H\mathbb{F}^{k \, *,*}_{p}(\Sigma^{\infty}_{+}X) \arrow[r, "i_{R \, *}"]
& (H\mathbb{F}^{k}_{p} \wedge H\mathbb{F}^{k}_{p})^{*,*}(\Sigma^{\infty}_{+}X) \arrow[d, "r_{*}"] \\
& \pi_{-*,-*}(\bigoplus\limits _{\alpha} \Sigma ^{p_{\alpha},q_{\alpha}}H\mathbb{F}^{k}_{p}) \otimes_{\pi_{-*,-*}H\mathbb{F}^{k}_{p}}H^{*,*}(X, \mathbb{F}_{p}).
\end{tikzcd}
\end{equation}
\end{definition}

\begin{proposition} Restricted to mod $p$ Chow groups, $\lambda_{X}$ preserves multiplication. \label{P:homringspec}
\end{proposition}
\begin{proof} Let $f: \Sigma^{\infty}_{+}X \to \Sigma^{2m,m}H\mathbb{F}_{p}^{k}$ and $g: \Sigma^{\infty}_{+}X \to \Sigma^{2n,n}H\mathbb{F}_{p}^{k}$ be given. We need to show that $\lambda_{X}(fg)=\lambda_{X}(f)\lambda_{X}(g).$ The right $H\mathbb{F}_{p}^{k}$ map $i_{R}$ is a morphism of commutative ring spectra. Hence, $i_{R *}$ is a homomorphism of rings. Hence, we need to prove that $r_{*}(i_{R *}(f)i_{R *}(g))=r_{*}(i_{R *}(f))r_{*}(i_{R *}(g)).$

\indent Applying the natural transformation $i^{*} \to i^{*}j_{*}j^{*}$ to $\mu^{D}_{2}$, we get a commuting diagram.
\begin{equation} \label{eq:ringhomspectra1diag}
\begin{tikzcd}
(H\mathbb{F}^{k}_{p}\wedge H\mathbb{F}^{k}_{p})\wedge(H\mathbb{F}^{k}_{p}\wedge H\mathbb{F}^{k}_{p}) \arrow[r, "\mu^{k}_{2}"] \arrow[d, "i^{*}\eta"]
& H\mathbb{F}^{k}_{p}\wedge H\mathbb{F}^{k}_{p} \arrow[d, "i^{*}\eta"] \\
i^{*}j_{*}((H\mathbb{F}^{K}_{p}\wedge H\mathbb{F}^{K}_{p})\wedge(H\mathbb{F}^{K}_{p}\wedge H\mathbb{F}^{K}_{p})) \arrow[r, "i^{*}j_{*}\mu^{K}_{2}"]
& i^{*}j_{*}(H\mathbb{F}^{K}_{p}\wedge H\mathbb{F}^{K}_{p}) \arrow[d, "\oplus \pi"] \\
& \bigoplus\limits _{\alpha} \Sigma ^{p_{\alpha},q_{\alpha}}H\mathbb{F}^{k}_{p}
\end{tikzcd}
\end{equation}
We will factor the left vertical morphism in this diagram. Consider the following triangle

\begin{equation} \label{eq:triangle}
\begin{tikzcd}
(\widehat{H}\mathbb{F}^{D}_{p}\wedge \widehat{H}\mathbb{F}^{D}_{p})\wedge(\widehat{H}\mathbb{F}^{D}_{p}\wedge \widehat{H}\mathbb{F}^{D}_{p}) \arrow[r, "\eta \wedge \eta"] \arrow[d, "\eta"]
& j_{*}(H\mathbb{F}^{K}_{p}\wedge H\mathbb{F}^{K}_{p})\wedge j_{*}(H\mathbb{F}^{K}_{p}\wedge H\mathbb{F}^{K}_{p}) \arrow[dl, ""] \\
j_{*}(H\mathbb{F}^{K}_{p}\wedge H\mathbb{F}^{K}_{p} \wedge H\mathbb{F}^{K}_{p}\wedge H\mathbb{F}^{K}_{p})
\end{tikzcd}
\end{equation}
 where the morphism on the hypotenuse is defined by the lax monoidal property of $j_{*}$. Note that the counit morphism $\epsilon:j^{*}j_{*} \to id.$ is an isomorphism since $j$ is open. By adjunction, the morphism on the hypotenuse of diagram \ref{eq:triangle} is induced by the isomorphism $$\epsilon \wedge \epsilon:j^{*}j_{*}(H\mathbb{F}^{K}_{p}\wedge H\mathbb{F}^{K}_{p})\wedge j^{*}j_{*}(H\mathbb{F}^{K}_{p}\wedge H\mathbb{F}^{K}_{p}) \to (H\mathbb{F}^{K}_{p}\wedge H\mathbb{F}^{K}_{p})\wedge (H\mathbb{F}^{K}_{p}\wedge H\mathbb{F}^{K}_{p}).$$ The morphism $\eta$ on the left leg of the triangle \ref{eq:triangle} is induced by the isomorphism $$j^{*}\eta:j^{*}((\widehat{H}\mathbb{F}^{D}_{p}\wedge \widehat{H}\mathbb{F}^{D}_{p})\wedge(\widehat{H}\mathbb{F}^{D}_{p}\wedge \widehat{H}\mathbb{F}^{D}_{p})) \to (H\mathbb{F}^{K}_{p}\wedge H\mathbb{F}^{K}_{p})\wedge(H\mathbb{F}^{K}_{p}\wedge H\mathbb{F}^{K}_{p}).$$ Using that pullback is strongly monoidal, we then have the following commuting triangle.
 
\[
\begin{tikzcd}
j^{*}(\widehat{H}\mathbb{F}^{D}_{p}\wedge \widehat{H}\mathbb{F}^{D}_{p})\wedge j^{*}(\widehat{H}\mathbb{F}^{D}_{p}\wedge \widehat{H}\mathbb{F}^{D}_{p}) \arrow[r, "j^{*}\eta \wedge j^{*}\eta"] \arrow[d, "j^{*}\eta"]
& j^{*}j_{*}(H\mathbb{F}^{K}_{p}\wedge H\mathbb{F}^{K}_{p})\wedge j^{*}j_{*}(H\mathbb{F}^{K}_{p}\wedge H\mathbb{F}^{K}_{p}) \arrow[dl, "\epsilon \wedge \epsilon"] \\
(H\mathbb{F}^{K}_{p}\wedge H\mathbb{F}^{K}_{p})\wedge(H\mathbb{F}^{K}_{p}\wedge H\mathbb{F}^{K}_{p})
\end{tikzcd}
\]
Thus, by adjunction, the triangle \ref{eq:triangle} commutes.

\indent Applying $i^{*}$ to triangle \ref{eq:triangle}, we then see that the commuting diagram \ref{eq:ringhomspectra1diag} is a sub-diagram of the commuting diagram 
\begin{equation} \label{eq:bigcd}
\begin{tikzcd}
(H\mathbb{F}^{k}_{p}\wedge H\mathbb{F}^{k}_{p})\wedge(H\mathbb{F}^{k}_{p}\wedge H\mathbb{F}^{k}_{p}) \arrow[r, "\mu^{k}_{2}"] \arrow[d, "i^{*}\eta \wedge i^{*}\eta"]
& H\mathbb{F}^{k}_{p}\wedge H\mathbb{F}^{k}_{p} \arrow[dd, "i^{*}\eta"] \\
i^{*}j_{*}(H\mathbb{F}^{K}_{p}\wedge H\mathbb{F}^{K}_{p})\wedge i^{*}j_{*}(H\mathbb{F}^{K}_{p}\wedge H\mathbb{F}^{K}_{p}) \arrow[d, ""]  \\
i^{*}j_{*}(H\mathbb{F}^{K}_{p}\wedge H\mathbb{F}^{K}_{p} \wedge H\mathbb{F}^{K}_{p}\wedge H\mathbb{F}^{K}_{p}) \arrow[r, "i^{*}j_{*}\mu^{K}_{2}"]
& i^{*}j_{*}(H\mathbb{F}^{K}_{p}\wedge H\mathbb{F}^{K}_{p}) \arrow[d, "\oplus \pi"] \\
& \bigoplus\limits _{\alpha} \Sigma ^{p_{\alpha},q_{\alpha}}H\mathbb{F}^{k}_{p}.
\end{tikzcd}
\end{equation}

\indent From diagram \ref{PsiDcommute}, $$(i^{*}\eta \wedge i^{*}\eta)\circ (\Psi_{k}\wedge \Psi_{k}):(\bigoplus\limits _{\alpha} \Sigma ^{p_{\alpha},q_{\alpha}}H\mathbb{F}^{k}_{p}) \wedge (\bigoplus\limits _{\alpha} \Sigma ^{p_{\alpha},q_{\alpha}}H\mathbb{F}^{k}_{p}) \to i^{*}j_{*}(H\mathbb{F}^{K}_{p}\wedge H\mathbb{F}^{K}_{p}) \wedge i^{*}j_{*}(H\mathbb{F}^{K}_{p}\wedge H\mathbb{F}^{K}_{p})$$ is equal to the composite $(i^{*}j_{*}\Psi_{K} \wedge i^{*}j_{*}\Psi_{K}) \circ (i^{*}\eta \wedge i^{*}\eta).$ Hence, diagram \ref{eq:bigcd} implies that the multiplication morphism $m=r \circ \mu_{2}^{k} \circ (\Psi_{k} \wedge \Psi_{k})$ on $$(\bigoplus\limits _{\alpha} \Sigma ^{p_{\alpha},q_{\alpha}}H\mathbb{F}^{k}_{p}) \wedge (\bigoplus\limits _{\alpha} \Sigma ^{p_{\alpha},q_{\alpha}}H\mathbb{F}^{k}_{p})$$ is equal to the following composite.

\begin{equation}   \label{compositeform}
\begin{tikzcd}
(\bigoplus\limits _{\alpha} \Sigma ^{p_{\alpha},q_{\alpha}}H\mathbb{F}^{k}_{p}) \wedge (\bigoplus\limits _{\alpha} \Sigma ^{p_{\alpha},q_{\alpha}}H\mathbb{F}^{k}_{p}) \arrow[d, "((i^{*}j_{*}\Psi_{K})\circ i^{*}\eta) \wedge ((i^{*}j_{*}\Psi_{K})\circ i^{*}\eta)"] \\
i^{*}j_{*}(H\mathbb{F}^{K}_{p}\wedge H\mathbb{F}^{K}_{p})\wedge i^{*}j_{*}(H\mathbb{F}^{K}_{p}\wedge H\mathbb{F}^{K}_{p}) \arrow[d, ""]  \\
i^{*}j_{*}(H\mathbb{F}^{K}_{p}\wedge H\mathbb{F}^{K}_{p} \wedge H\mathbb{F}^{K}_{p}\wedge H\mathbb{F}^{K}_{p}) \arrow[r, "i^{*}j_{*}\mu^{K}_{2}"]
& i^{*}j_{*}(H\mathbb{F}^{K}_{p}\wedge H\mathbb{F}^{K}_{p}) \arrow[d, "\oplus \pi"] \\
& \bigoplus\limits _{\alpha} \Sigma ^{p_{\alpha},q_{\alpha}}H\mathbb{F}^{k}_{p}
\end{tikzcd}
\end{equation}

\indent From Lemma \ref{lemmafornexthm}, the composites 

\[ 
\begin{tikzcd}
 \Sigma^{\infty}_{+}X \arrow[r, "f"]
 & \Sigma^{2m,m}H\mathbb{F}_{p}^{k} \arrow[r, "i_{R}"]
 & \Sigma^{2m,m}H\mathbb{F}^{k}_{p}\wedge H\mathbb{F}^{k}_{p} \arrow[r, "i^{*}\eta"]
 & \Sigma^{2m,m}i^{*}j_{*}(H\mathbb{F}^{K}_{p}\wedge H\mathbb{F}^{K}_{p}) \arrow[d, "i^{*}j_{*}\Psi^{-1}_{K}"]\\
 & & & \bigoplus\limits _{\alpha} \Sigma^{p_{\alpha}+2m,q_{\alpha}+m}i^{*}j_{*}H\mathbb{F}^{K}_{p} \arrow[d, "\oplus \pi_{0}"] \\
 & & & \bigoplus\limits _{\alpha} \Sigma^{p_{\alpha}+2m-1,q_{\alpha}+m-1}H\mathbb{F}^{k}_{p}
\end{tikzcd}
\]

and 

\[ 
\begin{tikzcd}
 \Sigma^{\infty}_{+}X \arrow[r, "g"]
 & \Sigma^{2n,n}H\mathbb{F}_{p}^{k} \arrow[r, "i_{R}"]
 & \Sigma^{2n,n}H\mathbb{F}^{k}_{p}\wedge H\mathbb{F}^{k}_{p} \arrow[r, "i^{*}\eta"]
 & \Sigma^{2n,n}i^{*}j_{*}(H\mathbb{F}^{K}_{p}\wedge H\mathbb{F}^{K}_{p}) \arrow[d, "i^{*}j_{*}\Psi^{-1}_{K}"]\\
 & & & \bigoplus\limits _{\alpha} \Sigma^{p_{\alpha}+2n,q_{\alpha}+n}i^{*}j_{*}H\mathbb{F}^{K}_{p} \arrow[d, "\oplus \pi_{0}"] \\
 & & & \bigoplus\limits _{\alpha} \Sigma^{p_{\alpha}+2n-1,q_{\alpha}+n-1}H\mathbb{F}^{k}_{p}
\end{tikzcd}
\]

are equal to $0$. It follows that $i^{*}\eta \circ r \circ i_{R} \circ f=i^{*}\eta \circ i_{R}\circ f$ and $i^{*}\eta \circ r \circ i_{R} \circ g=i^{*}\eta \circ i_{R}\circ g$ in the following two diagrams.
\begin{equation} \label{diagX1}
\begin{tikzcd}
& \Sigma^{\infty}_{+}X  \arrow[d, "f"] \\
& \Sigma^{2m,m}H\mathbb{F}_{p}^{k} \arrow[d, "i_{R}"] \\ 
& \Sigma^{2m,m}H\mathbb{F}^{k}_{p}\wedge H\mathbb{F}^{k}_{p} \arrow[dl, "r"] \arrow[d, "i^{*}\eta"] \\
\bigoplus\limits _{\alpha} \Sigma^{p_{\alpha}+2m,q_{\alpha}+m}H\mathbb{F}^{k}_{p} \arrow[r, "i^{*}\eta"]
& \Sigma^{2m,m}i^{*}j_{*}(H\mathbb{F}^{K}_{p}\wedge H\mathbb{F}^{K}_{p})
\end{tikzcd}
\end{equation} 

\begin{equation} \label{diagX2}
\begin{tikzcd}
& \Sigma^{\infty}_{+}X  \arrow[d, "g"] \\
& \Sigma^{2n,n}H\mathbb{F}_{p}^{k} \arrow[d, "i_{R}"] \\ 
& \Sigma^{2n,n}H\mathbb{F}^{k}_{p}\wedge H\mathbb{F}^{k}_{p} \arrow[dl, "r"] \arrow[d, "i^{*}\eta"] \\
\bigoplus\limits _{\alpha} \Sigma^{p_{\alpha}+2n,q_{\alpha}+n}H\mathbb{F}^{k}_{p} \arrow[r, "i^{*}\eta"]
& \Sigma^{2n,n}i^{*}j_{*}(H\mathbb{F}^{K}_{p}\wedge H\mathbb{F}^{K}_{p})
\end{tikzcd}
\end{equation}

\indent To show that $r_{*}(i_{R *}(f)i_{R *}(g))=r_{*}(i_{R *}(f))r_{*}(i_{R *}(g))$, we consider the following commuting diagram where $\Delta$ is the diagonal morphism.

\begin{equation}  \label{diagramX}
\begin{tikzcd}
 \Sigma^{\infty}_{+}X \arrow[d, "\Delta"] \\
 \Sigma^{\infty}_{+}X \wedge \Sigma^{\infty}_{+}X \arrow[d, "f \wedge g"] \\
  \Sigma^{2m,m}H\mathbb{F}_{p}^{k} \wedge \Sigma^{2n,n}H\mathbb{F}_{p}^{k} \arrow[d, "i_{R}\wedge i_{R}"]\\
 (\Sigma^{2m,m}H\mathbb{F}^{k}_{p}\wedge H\mathbb{F}^{k}_{p})\wedge(\Sigma^{2n,n}H\mathbb{F}^{k}_{p}\wedge H\mathbb{F}^{k}_{p}) \arrow[d, "i^{*}\eta \wedge i^{*}\eta"] \arrow[r, "\mu_{2}^{k}"] & \Sigma^{2m,m}H\mathbb{F}^{k}_{p}\wedge \Sigma^{2n,n}H\mathbb{F}^{k}_{p} \arrow[dd, "i^{*}\eta"] \\
 i^{*}j_{*}(\Sigma^{2m,m}H\mathbb{F}^{K}_{p}\wedge H\mathbb{F}^{K}_{p})\wedge i^{*}j_{*}(\Sigma^{2n,n}H\mathbb{F}^{K}_{p}\wedge H\mathbb{F}^{K}_{p}) \arrow[d, ""]\\
 i^{*}j_{*}(\Sigma^{2(m+n),m+n}H\mathbb{F}^{K}_{p}\wedge H\mathbb{F}^{K}_{p} \wedge H\mathbb{F}^{K}_{p}\wedge H\mathbb{F}^{K}_{p}) \arrow[r, "i^{*}j_{*}\mu_{2}^{K}"]
& i^{*}j_{*}(\Sigma^{2(m+n),m+n}H\mathbb{F}^{K}_{p}\wedge H\mathbb{F}^{K}_{p}) \arrow[d, "\oplus \pi"] \\
 & \bigoplus\limits _{\alpha} \Sigma ^{p_{\alpha}+2m+2n,q_{\alpha}+m+n}H\mathbb{F}^{k}_{p}
\end{tikzcd}
\end{equation}
The composite $\oplus \pi \circ i^{*}\eta \circ \mu_{2}^{k} \circ (i_{R}\wedge i_{R}) \circ (f \wedge g)\circ \Delta$ in this diagram is equal to $r_{*}(i_{R *}(f)i_{R *}(g))$. From diagrams \ref{diagX1} and \ref{diagX2}, the composite given by 

\begin{equation}  \label{lastidagcoaction}
\begin{tikzcd}
 \Sigma^{\infty}_{+}X \arrow[d, "\Delta"] \\
 \Sigma^{\infty}_{+}X \wedge \Sigma^{\infty}_{+}X \arrow[d, "f \wedge g"] \\
  \Sigma^{2m,m}H\mathbb{F}_{p}^{k} \wedge \Sigma^{2n,n}H\mathbb{F}_{p}^{k} \arrow[d, "i_{R}\wedge i_{R}"]\\
 (\Sigma^{2m,m}H\mathbb{F}^{k}_{p}\wedge H\mathbb{F}^{k}_{p})\wedge(\Sigma^{2n,n}H\mathbb{F}^{k}_{p}\wedge H\mathbb{F}^{k}_{p}) \arrow[d, "r \wedge r"] \\
 (\bigoplus\limits _{\alpha} \Sigma^{p_{\alpha}+2m,q_{\alpha}+m}H\mathbb{F}^{k}_{p}) \wedge (\bigoplus\limits _{\alpha} \Sigma^{p_{\alpha}+2n,q_{\alpha}+n}H\mathbb{F}^{k}_{p}) \arrow[d, "i^{*}\eta \wedge i^{*}\eta"]\\ 
 i^{*}j_{*}(\Sigma^{2m,m}H\mathbb{F}^{K}_{p}\wedge H\mathbb{F}^{K}_{p})\wedge i^{*}j_{*}(\Sigma^{2n,n}H\mathbb{F}^{K}_{p}\wedge H\mathbb{F}^{K}_{p}) \arrow[d, ""]\\
 i^{*}j_{*}(\Sigma^{2(m+n),m+n}H\mathbb{F}^{K}_{p}\wedge H\mathbb{F}^{K}_{p} \wedge H\mathbb{F}^{K}_{p}\wedge H\mathbb{F}^{K}_{p}) \arrow[r, "i^{*}j_{*}\mu_{2}^{K}"]
& i^{*}j_{*}(\Sigma^{2(m+n),m+n}H\mathbb{F}^{K}_{p}\wedge H\mathbb{F}^{K}_{p}) \arrow[d, "\oplus \pi"] \\
 & \bigoplus\limits _{\alpha} \Sigma ^{p_{\alpha}+2m+2n,q_{\alpha}+m+n}H\mathbb{F}^{k}_{p}
\end{tikzcd}
\end{equation}
is equal to the composite given by diagram \ref{diagramX}. From diagram \ref{compositeform}, the composite given by diagram \ref{lastidagcoaction} is equal to $\Sigma^{2(m+n),m+n}m\circ (r\wedge r)\circ (i_{R}\wedge i_{R})\circ (f \wedge g) \circ \Delta=r_{*}(i_{R *}(f))r_{*}(i_{R *}(g))$. Thus, $r_{*}(i_{R *}(f)i_{R *}(g))=r_{*}(i_{R *}(f))r_{*}(i_{R *}(g))$ as desired.

\end{proof}

\section{Cartan formula} \label{sectioncartan}

\indent In this section, we use the coaction map constructed in the previous section to prove a Cartan formula for the operations $P^{n}_{k}$ restricted to mod $p$ Chow groups. Let $X \in \textup{Sm}_{k}$. Let $\langle \cdot,\cdot \rangle$ denote the pairing between $\mathcal{A}^{k}_{*,*}$ and $H\mathbb{F}^{k \, *,*}_{p}H\mathbb{F}^{k}_{p}$. Let $n \geq 0$. For $x \in H^{*,*}(X,\mathbb{F}_{p})$ with $\lambda_{X}(x)= \Sigma y_{i} \otimes x_{i}$, we have $P^{n}_{k}(x)=\Sigma \langle y_{i},P^{n}_{k}\rangle x_{i}$.

\begin{proposition}
Let $x,y\in CH^{*}(X)/p$ and $i \geq0$. Then $$P^{i}_{k}(xy)=\sum_{j=0}^{i}P^{j}_{k}(x)P^{i-j}_{k}(y).$$ 
\end{proposition}
\begin{proof}
From the definition of $P^{i}_{k}$, $\langle \xi_{1}^{i}, P^{i}_{k}\rangle=1$ and $\langle\omega(\alpha), P^{i}_{k}\rangle=0$ for all sequences $\alpha \neq (0,i,0,0,\ldots)$. Using the coaction map \ref{coacitonmapdef}, we write $$\lambda_{X}(x)=\sum_{q} \omega(\alpha^{1}_{q}) \otimes x_{q}$$ and $$\lambda_{X}(y)=\sum_{r} \omega(\alpha^{2}_{r}) \otimes y_{r}$$ for some sequences $\alpha^{1}_{q}, \alpha^{2}_{r}$. Then $$\lambda_{X}(xy)=\sum_{q, r} ((\omega(\alpha^{1}_{q}) \omega(\alpha^{2}_{r}) \otimes x_{q}y_{r}).$$ For any $2$ sequences $\alpha^{1}_{q}, \alpha^{2}_{r}$ appearing in these sums, we have $\omega(\alpha^{1}_{q}) \omega(\alpha^{2}_{r})=0$ if the relation $\tau_{m}^{2}=0$ from Proposition \ref{P:tau2=0} applies for some $m \geq0$, or else $\omega(\alpha^{1}_{q}) \omega(\alpha^{2}_{r})=\pm \omega(\alpha^{1}_{q}+\alpha^{2}_{r})$.

\indent From the definition of $\lambda_{X}$, $$P^{i}_{k}(xy)=\sum_{q, r} \langle (\omega(\alpha^{1}_{q}) \omega(\alpha^{2}_{r}), P^{i}_{k}\rangle x_{q}y_{r}.$$ Proposition \ref{P:tau2=0} implies that if $\omega(\alpha_{1}) \omega(\alpha_{2})=a\xi_{1}^{i}$ for two sequences $\alpha_{1}, \alpha_{2}$ and $a\neq 0 \in H^{*,*}(k, \mathbb{F}_{p})$, then $a=1$ and $\omega(\alpha_{1})=\xi_{1}^{j}, \, \omega(\alpha_{2})=\xi_{1}^{i-j}$ for some $0 \leq j \leq i$. As $P^{i}_{k}$ is dual to $\xi_{1}^{i}$, the only terms for which $\langle\omega(\alpha^{1}_{q}+\alpha^{2}_{r}), P^{i}_{k}\rangle\neq 0$ are 
of the form $\omega(\alpha^{1}_{q_{j}})=\xi_{1}^{j}, \omega(\alpha^{2}_{r_{j}})=\xi_{1}^{i-j}$ for $0 \leq j \leq i$. Hence, 
\begin{equation} \label{eq:cartanform1} P^{i}_{k}(xy)=\sum_{j=0}^{i} \langle\omega(\alpha^{1}_{q_{j}}+\alpha^{2}_{r_{j}}), P^{i}_{k}\rangle x_{q_{j}}y_{r_{j}}=\sum_{j=0}^{i}P^{j}_{k}(x)P^{i-j}_{k}(y) \end{equation} 
as required.

\end{proof}

\section{$p$th power and instability}

\indent In this section, for $n \in \mathbb{N}$, we prove that $P^{n}_{k}$ is the $p$th power on $CH^{n}(-)/p$. Letting $f:\textup{Spec}(k) \to \textup{Spec}(\mathbb{F}_{p})$ denote the structure map, it suffices to prove that $f^{*}(P^{n}_{\mathbb{F}_{p}})(\iota_{n})=\iota^{p}_{n}$ for the canonical element $\iota_{n} \in H^{2n,n}(K_{n,k},\mathbb{F}_{p})$ where $K_{n,k} \in H(k)$ is the motivic Eilenberg-MacLane space representing $H^{2n,n}(-, \mathbb{F}_{p})$. Our proof makes use of Morel's $S^1$-recognition principle. 

\indent We refer to \cite[Section 3]{EHK+} as a reference for the $S^1$-recognition principle. For a base scheme $S$, let $\textrm{PSh}_{\textrm{nis}}(\textrm{Sm}_{S})$ denote the category of Nisnevich local presheaves of spaces on $\textrm{Sm}_{S}$. The unstable motivic homotopy category $H(S)$ can be described as the full subcategory of $\textrm{PSh}_{\textrm{nis}}(\textrm{Sm}_{S})$ of presheaves that are $\mathbb{A}^{1}$-invariant. Let $L_{\textrm{mot}}:\textrm{PSh}_{\textrm{nis}}(\textrm{Sm}_{S}) \to H(S)$ denote the $\mathbb{A}^{1}$-localization functor. Let $SH^{S^{1}}(S)$ denote the  stable motivic homotopy category of $S^{1}$-spectra. For a morphism $f: S_{1} \to S_{2}$ of base schemes, we have the adjoint functors of pullback $f^{*} \coloneqq Lf^{*}$ and pushforward $f_{*} \coloneqq Rf_{*}$: $$f^{*}:H(S_{2}) \rightleftarrows H(S_{2}): f_{*}.$$ For $f: S_{1} \to S_{2}$ smooth, $f^{*}$ admits a left adjoint $f_{\#}$ such that $f_{\#}(X)=X \in H(S_{2})$ for any $X \in \textrm{Sm}_{S_{1}}.$

\indent For $C=\textrm{PSh}_{\textrm{nis}}(\textrm{Sm}_{S})$ or $H(S)$, we consider the $n$-fold bar constructions $\textrm{B}^{n}_{C}$ that are adjoint to the $n$th $S^1$-deloopings $\Omega^{n}$:$$\textrm{B}^{n}_{C}: \textrm{Mon}_{\mathcal{E}_{n}}(C) \rightleftarrows C: \Omega^{n}.$$ For $C=\textrm{PSh}_{\textrm{nis}}(\textrm{Sm}_{S})$ or $H(S)$, we let $\textrm{Stab}(C) \coloneqq C \otimes \textrm{Spt}$ denote the $S^1$-stabilization of $C$.  We also consider the infinite bar construction $$\textrm{B}_{C}^{\infty}: \textrm{CMon}(C)=\textrm{Mon}_{\mathcal{E}_{\infty}}(C)\rightleftarrows \textrm{Stab}(C): \Omega^{\infty}.$$ For $C=\textrm{PSh}_{\textrm{nis}}(\textrm{Sm}_{S})$, we denote $\textrm{B}^{n}_{C}$ by $\textrm{B}^{n}_{\textrm{nis}}$ and we denote $\textrm{B}^{\infty}_{C}$ by $\textrm{B}^{\infty}_{\textrm{nis}}$. Similarly, for $C=H(S)$, we denote $\textrm{B}^{n}_{C}$ by $\textrm{B}^{n}_{\textrm{mot}}$ and we denote $\textrm{B}^{\infty}_{C}$ by $\textrm{B}^{\infty}_{\textrm{mot}}$. For later use, we note that $\textrm{B}^{n}_{\textrm{nis}}$ and $\textrm{B}^{\infty}_{\textrm{nis}}$ commute with pullbacks.

\begin{definition} Define $X \in \textup{Mon}(H(S))$ to be strongly $\mathbb{A}^{1}$-invariant if $\textup{B}_{\textup{nis}}X \simeq \textup{B}_{\textup{mot}}X$. Define $X \in \textup{CMon}(H(S))$ to be strictly $\mathbb{A}^{1}$-invariant if $\textup{B}^{n}_{\textup{nis}}X \simeq \textup{B}^{n}_{\textup{mot}}X$ for all $n\geq0$.
\end{definition}

\indent Most of the proof of the following proposition was suggested to us by Marc Hoyois.

\begin{proposition} \label{T:pullbaclEB}
Let $k$ be a perfect field of characteristic $p$ and let $i:\textrm{Spec}(k) \to \textrm{Spec}(D)$ be a closed embedding where $D$ is a complete unramified DVR with generic point $j: \textrm{Spec}(K) \to \textrm{Spec}(D)$. Fix $n >0$. We let $K_{n,D} \coloneqq \Omega^{\infty}_{\mathbb{P}^{1}}\Sigma^{2n,n}\widehat{H}\mathbb{F}^{D}_{p}$. Then the morphism $i^{*}K_{n,D} \to K_{n,k}$ induced by $i^{*} \Sigma^{2n,n}\widehat{H}\mathbb{F}_{p}^{D} \cong \Sigma^{2n,n}H\mathbb{F}_{p}^{k}$ is an isomorphism in $H(k)$.
\end{proposition}
\begin{proof}
We first prove that $K_{n, D}$ is connected. Let $R$ be a Henselian local ring that is essentially smooth over $D$. From \cite[Corollary 4.2]{Gei}, the Bloch-Levine Chow groups $CH^{m}(R)$ of $R$ vanish for $m \geq 1.$ Thus, $\pi_{0}^{\textup{nis}}(K_{n,D}(\textrm{Spec}(R))) \simeq * $ since $K_{n,D} \in H(D)$ represents the codimension $n$ mod $p$ Bloch-Levine Chow group. 

\indent Now we prove that $i^{*}K_{n, D}$ is connected. As $j:\textrm{Spec}(K) \to \textrm{Spec}(D)$ is smooth, $j^{*}K_{n,D} \simeq K_{n,K}$. Consider the homotopy pushout $P$ in $\textrm{PSh}_{\textrm{nis}}(\textrm{Sm}_{D})$ of the following diagram.
\[
\begin{tikzcd}
   j_{\#}K_{n,K} \arrow[r, ""] \arrow[d, ""]
   & K_{n,D}  \\
   \textrm{Spec}(K) 
 \end{tikzcd}
\]
The morphism $j_{\#}K_{n,K} \to \textrm{Spec}(K)$ induces a bijection on $\pi^{\textrm{nis}}_{0}$. Hence, $\pi^{\textrm{nis}}_{0}(K_{n,D}) \simeq \pi^{\textrm{nis}}_{0}(P)$. From the gluing square \cite[Theorem 2.21]{MorVoe}, $L_{\textrm{mot}}(P) \simeq i_{*}i^{*}(K_{n,D})$. From \cite[Corollary 3.22]{MorVoe}, it follows that $i_{*}i^{*}(K_{n,D})$ is connected since $K_{n,D}$ is connected. Let $k \to S_{k}$ be an essentially smooth homomorphism of rings where $S_{k}$ is Henselian local. The ring $S_{k}$ admits a lift $S_{D}$ where $D \to S_{D}$ is essentially smooth and $S_{D}$ is Henselian local. Hence, $i^{*}(K_{n,D})(S_{k})\simeq i_{*}i^{*}(K_{n,D})(S_{D})$ is connected. Thus, $i^{*}K_{n,D} \in H(k)$ is connected. In particular, $\pi^{\textrm{nis}}_{0}(i^{*}(K_{n,D}))$ is strongly $\mathbb{A}^{1}$-invariant. The $S^{1}$-recognition principle \cite[Theorem 3.1.12]{EHK+} then implies that $i^{*}K_{n,D}$ is strictly $\mathbb{A}^{1}$-invariant. Note that $K_{n,k}$ is also strictly $\mathbb{A}^{1}$-invariant since $\pi^{\textrm{nis}}_{0}(K_{n,k})$ is strongly $\mathbb{A}^{1}$-invariant.

\indent 
From \cite[Theorem 8.18]{Spi}, we have $$\textrm{B}^{\infty}_{\textup{mot}}i^{*}(K_{n,D})\cong i^{*}(\textrm{B}^{\infty}_{\textup{mot}}K_{n,D}) \cong i^{*}(\Omega^{\infty}_{\mathbb{G}_{m}} \Sigma^{2n,n}\widehat{H}\mathbb{F}^{D}_{p}) \cong \Omega^{\infty}_{\mathbb{G}_{m}} \Sigma^{2n,n}H\mathbb{F}^{k}_{p} \cong \textrm{B}^{\infty}_{\textup{mot}}K_{n,k}$$ in $SH^{S^{1}}(k)$. Then \cite[Corollary 3.1.15]{EHK+} implies that $i^{*}K_{n,D} \cong K_{n,k}$ in $H(k)$.
\end{proof}

\begin{proposition} \label{P:pthpower}
Let $k$ be a field of characteristic $p$ with structure map $f:\textup{Spec}(k)\to \textup{Spec}(\mathbb{F}_{p})$ and let $\iota_{n} \in H^{2n,n}(K_{n,k}, \mathbb{F}_{p})$ be the canonical element. Then $f^{*}P^{n}_{\mathbb{F}_{p}}(\iota_{n})=\iota^{p}_{n}$.
\end{proposition}

\begin{proof}
First, we assume that $k$ is perfect. Let $D$ be a DVR having $k$ as a residue field with inclusion morphism $i:\textrm{Spec}(k) \to \textrm{Spec}(D)$ and generic point $j:\textrm{Spec}(K)\to \textrm{Spec}(D)$. From Proposition \ref{T:pullbaclEB}, $i^{*}K_{n,D} \cong K_{n,k}$. Over all base schemes $S$, we let $\iota_{n}$ denote the canonical element in $H^{2n,n}(K_{n,S}, \mathbb{F}_{p})$. Apply $i^{*} \to i^{*}j_{*}j^{*}$ to the natural morphism $\iota_{n}:\Sigma^{\infty}_{+}K_{n,D} \to \Sigma^{2n,n}\widehat{H}\mathbb{F}_{p}^{D}$ to get the following commuting square.

\[
\begin{tikzcd}
\Sigma^{\infty}_{+}K_{n,k} \arrow[r, "\iota_{n}"] \arrow[d, "i^{*}\eta"]
 & \Sigma^{2n,n}H\mathbb{F}_{p}^{k} \arrow[d, "i^{*}\eta"]\\
 i^{*}j_{*}\Sigma^{\infty}_{+}K_{n,K} \arrow[r, "i^{*}j_{*}\iota_{n}"] & \Sigma^{2n,n}i^{*}j_{*}H\mathbb{F}_{p}^{K}
 \end{tikzcd}
 \]  Apply $i^{*}\eta:i^{*} \to i^{*}j_{*}j^{*}$ to the morphism $\Sigma^{\infty}_{+}K_{n,D} \to \Sigma^{2pn,pn} \widehat{H}\mathbb{F}_{p}^{D}$ in $SH(D)$ corresponding to $\iota^{p}_{n}$ to get the commutative diagram 

\begin{equation}           \label{diagrampthpower}
\begin{tikzcd} 
\Sigma^{\infty}_{+}K_{n,k} \arrow[r, "\iota^{p}_{n}"] \arrow[d, "i^{*}\eta"]
& \Sigma^{2pn,pn} H\mathbb{F}_{p}^{k} \arrow[d, "i^{*}\eta"]\\
i^{*}j_{*}\Sigma^{\infty}_{+}K_{n,K} \arrow[r, "i^{*}j_{*}\iota^{p}_{n}"]
& i^{*}j_{*}\Sigma^{2pn,pn} H\mathbb{F}_{p}^{K} \arrow[r, "\pi"]& \Sigma^{2pn,pn} H\mathbb{F}_{p}^{k} .
\end{tikzcd}
\end{equation}
From \cite[Lemma 9.8]{Voe}, $i^{*}j_{*}\iota^{p}_{n}=i^{*}j_{*}P^{n}_{K}(\iota_{n})$. Hence, we can rewrite the bottom row of \ref{diagrampthpower} as 
\[
\begin{tikzcd}
i^{*}j_{*}\Sigma^{\infty}_{+}K_{n,K} \arrow[r,"i^{*}j_{*}\iota_{n}"] 
& i^{*}j_{*}\Sigma^{2n,n}H\mathbb{F}_{p}^{K} \arrow[r, "i^{*}j_{*}P^{n}_{K}"] & i^{*}j_{*}\Sigma^{2pn,pn}H\mathbb{F}_{p}^{K} \arrow[r, "\pi"]& \Sigma^{2pn,pn} H\mathbb{F}_{p}^{k}.
\end{tikzcd}
\]

From Theorem \ref{T:phiproperties} and the above commuting diagrams, $P^{n}_{k}(\iota_{n})=\pi \circ (i^{*}j_{*}P^{n}_{K})\circ (i^{*}j_{*}\iota_{n})\circ i^{*}\eta$. Hence, from diagram \ref{diagrampthpower}, we get $P^{n}_{k}(\iota_{n})=\pi \circ (i^{*}j_{*}\iota_{n}^{p})\circ i^{*}\eta=\iota^{p}_{n}$.

\indent For $k$ not perfect, we have an essentially smooth morphism $f: \textrm{Spec}(k) \to \textrm{Spec}(\mathbb{F}_{p})$ and $f^{*}(K_{n,\mathbb{F}_{p}})\cong K_{n,k}$ \cite[Theorem 2.11]{HKO}. As $\mathbb{F}_{p}$ is perfect, we then have $f^{*}(P^{n}_{\mathbb{F}_{p}}(\iota_{n}))=f^{*}(P^{n}_{\mathbb{F}_{p}})(\iota_{n})=f^{*}(\iota_{n}^{p})=\iota^{p}_{n}.$
\end{proof}

\indent From Proposition \ref{P:pullback}, we have the following corollary.

\begin{corollary} Let $X \in \textup{Sm}_{k}$. Then $P^{n}_{k}$ is the $p$th power on $CH^{n}(X)/p$.
\end{corollary}

\indent Now that we know $f^{*}(P^{n}_{\mathbb{F}_{p}})$ is the $p$th power on $H^{2n,n}(-, \mathbb{F}_{p})$ for all $n \geq 1$, we can prove an instability result. Let $f:\textup{Spec}(k)\to \textup{Spec}(\mathbb{F}_{p})$ be the structure morphism.

\begin{proposition}
Let $p, q, n \geq 0$ be integers such that $n>p-q$ and $n \geq q$. Let $X \in H(k)$ and let $x \in H^{p,q}(X, \mathbb{F}_{p})$. Then $f^{*}(P^{n}_{\mathbb{F}_{p}})(x)=0$.
\end{proposition}

\begin{proof}

Voevodsky's proof in \cite[Lemma 9.9]{Voe} works here since $f^{*}(P^{n}_{\mathbb{F}_{p}})$ is the $p$th power on $H^{2n,n}(-, \mathbb{F}_{p})$ by Proposition \ref{P:pthpower}.
\end{proof}

\begin{corollary} 
Let $X \in \textup{Sm}_{k}$. Then $P^{n}_{k}$ is the $0$ map on $CH^{m}(X)/p$ for $m<n$.
\end{corollary}
 \section{Proper pushforward}
 \indent In this section, we restrict our attention to mod $p$ Chow groups on $\text{Sm}_{k}$. The ring of mod $p$ Chow groups is an oriented cohomology pretheory in the sense of \cite[Section 1]{Pan} with perfect integration given by proper pushforward on Chow groups. Consider the total cohomological Steenrod operation $P_{k} \coloneqq P^{0}_{k}+P^{1}_{k}+P^{2}_{k}+ \cdots:CH^{*}(-)/p \to CH^{*}(-)/p.$ From the Cartan formula \ref{sectioncartan}, $P_{k}$ is a ring morphism of oriented cohomology pretheories in the sense of \cite[Definition 1.1.7]{Pan}.
 
 \indent Let $\mathbb{Z}[[c_{1},c_{2}, \ldots]]$ denote the power series ring on Chern classes $c_{i}$ for $i \geq 1$ and let $w \in \mathbb{Z}[[c_{1},c_{2}, \ldots]]$ denote the total characteristic class corresponding to the polynomial $f(x)=1+x^{p-1}.$ For $p=2$, $w$ is the total Chern class.  Let $X \in \textup{Sm}_{k}$. For a line bundle $L$ on $X$, $w(L)=1+c_{1}^{p-1}(L) \in CH^{*}(X)$. For a vector bundle $V$ on $X$ that has a filtration by subbundles with quotients given by line bundles $L_{1}, \ldots, L_{m}$, $w(V)=w(L_{1})\cdots w(L_{m}).$ Let $w_{i}$ denote the $i$th homogeneous component of $w$ for $i \geq 0$. We have $w_{i}=0$ if $p-1$ does not divide $i$. Define the total homological Steenrod operation $P^{X} \coloneqq w(-T_{X})\circ P_{k}:CH^{*}(X)/p \to CH^{*}(X)/p$ where $T_{X}$ is the tangent bundle on $X$. For $i \geq0$, let $P^{X}_{i}$ denote the $(p-1)i$th homogeneous component of $P^{X}$. The following proposition is a consequence of the general Riemann-Roch formulas proved by Panin in \cite{Pan}
 
 \begin{proposition} \label{pushforward}
 Let $f:X \to Y$ be a morphism of smooth projective varieties over $k$. Then 
 \[
 \begin{tikzcd}
 CH^{*}(X)/p \arrow[d, "f_{*}"]  \arrow[r, "P^{X}"] 
 & CH^{*}(X)/p \arrow[d, "f_{*}"]\\
  CH^{*}(Y)/p \arrow[r, "P^{Y}"] & CH^{*}(Y)/p
  \end{tikzcd}
  \]
  commutes.
  \end{proposition}
  
  \begin{proof}
  This is \cite[Theorem 2.5.4]{Pan}. See \cite[Section 2.6]{Pan} for a discussion relevant to our situation. The main ingredients are that the operations $P^{n}_{k}$ satisfy the Cartan formula and that $P^{n}_{k}$ is the $p$th power on $CH^{n}(-)/p$.
  \end{proof}
  
\indent Restricting to the case $p=\textup{char}(k)=2$, we obtain a Wu formula from the work of Panin \cite[Theorem 2.5.3]{Pan}. Here, $w=c$ is the total Chern class and we let $\textup{Sq}$ denote the total Steenrod square $P_{k}$ on $CH^{*}(-)/2$.

\begin{proposition}
Let $X, Y$ be smooth projective varieties over $k$, and let $i: X \xhookrightarrow{} Y$ be a closed embedding with normal bundle $N$. Then $$i_{*}(c(N))=\textup{Sq}([X])$$ in $CH^{*}(Y)/2$ where $[X] \in CH^{*}(Y)/2$ denotes the mod $2$ cycle class of $X$.

\end{proposition}

\section{Rost's degree formula}    \label{Rost}
\indent Now that we have Steenrod operations on mod $p$ Chow groups of $\textup{Sm}_{k}$, we can prove Rost's degree formula \cite[Theorem 6.4]{Mer} without any restrictions on the characteristic of the base field. We closely follow the presentation of Merkurjev \cite{Mer} where Steenrod operations (assuming restrictions on the characteristic of the base field) are used to prove degree formulas. In \cite{Hau3}, Haution extended the Rost degree formulas to base fields of characteristic $2$. 

\indent For a variety $X$ over $k$, let $n_{X}$ denote the greatest common divisor of $\textup{deg}(x)$ over all closed points $x \xhookrightarrow{} X$. Let $X \in \textup{Sm}_{k}$ be projective of dimension $d>0$. Applying Proposition \ref{pushforward} to the structure morphism $X \to \textup{Spec}(k)$ and $[X] \in CH_{d}(X)/p$, we see that $p \mid \textup{deg}(w_{d}(-T_{X}))$. 

\begin{proposition} Let $f:X \to Y$ be a morphism of projective varieties $X, Y \in \textup{Sm}_{k}$ of dimension $d>0$. Then $n_{Y} \mid n_{X}$ and $$\frac{\textup{deg}(w_{d}(-T_{X}))}{p} \equiv \textup{deg}(f)\cdot \frac{\textup{deg}(w_{d}(-T_{Y}))}{p} \mod n_{Y}.$$
\end{proposition}
 \begin{proof}
 The proof in \cite[Theorem 6.4]{Mer} works here. From Proposition \ref{pushforward}, $f_{*}(w_{d}(-T_{X}))\equiv \textup{deg}(f)w_{d}(-T_{Y}) \in CH_{0}(Y)/p.$ We then take the degree homomorphism to finish the proof.
 \end{proof}

\section{Specialization map}

\indent Fix a complete unramified DVR $D$ with residue field $i:\textup{Spec}(k) \to \textup{Spec}(D)$ and fraction field $j:\textup{Spec}(K) \to \textup{Spec}(D)$ as before. Let $X \in \textup{Sm}_{D}$ be projective with special fiber $X_{k}$ and generic fiber $X_{K}$. As described in \cite[Chapter 20.3]{Ful}, there are specialization maps $\sigma_{n}:CH^{n}(X_{K}) \to CH^{n}(X_{k})$ defined for all $n \geq 0$. The specialization maps can be defined at the level of cycles. Namely, for an irreducible closed subvariety $Z_{K} \subset X_{K}$ of codimension $n$, we let $Z_{k}$ denote the special fiber of the reduced closed subscheme $\overline{Z_{K}} \subset X$ associated to $Z_{K} \subset X$. Then $\sigma_{n}(\left<Z_{K} \right>)=\left<Z_{k} \right> \in CH^{n}(X_{k})$. We also let $\sigma_{n}$ denote the specialization map induced on mod $p$ Chow groups.

\indent We now show that the Steenrod operations $P^{n}_{k}$ defined on $CH^{*}(X_{k})$ are compatible with the operations $P^{n}_{K}$ defined on $CH^{*}(X_{K})$.

\begin{proposition} \label{P:specialization}
Let $m \geq 0$ and let $Z_{K} \subset X_{K}$ be a closed subvariety of codimension $n$. Let $\left<Z_{K} \right> \in CH^{n}(X_{K})/p$ denote the mod $p$ cycle class of $Z_{K}$. Then $$P^{m}_{k}(\sigma_{n}(\left<Z_{K} \right>))=\sigma_{n+m(p-1)}(P^{m}_{K}(\left<Z_{K} \right>)) \in CH^{n+m(p-1)}(X_{k})/p.$$
\end{proposition}

\begin{proof}
The mod $p$ cycle class of $\overline{Z_{K}} \subset X$ induces a map $$f_{D}:\Sigma^{\infty}_{+}X \to \Sigma^{2n,n}\widehat{H}\mathbb{F}_{p}^{D}$$ in $SH(D)$. The map $i^{*}f_{D}$ gives the mod $p$ cycle class of $Z_{k}$ (the special fiber of $\overline{Z_{K}} \subset X$) and $j^{*}f_{D}$ gives the mod $p$ cycle class of $Z_{K}$. Applying the natural transformation $i^{*}\eta:i^{*} \to i^{*}j_{*}j^{*}$ to $f_{D}$, we get a commuting square. 

\begin{equation} \label{diag1special}
\begin{tikzcd}
\Sigma^{\infty}_{+}X_{k} \arrow[r, "i^{*}f_{D}"] \arrow[d, "i^{*}\eta"] & \Sigma^{2n,n}H\mathbb{F}_{p}^{k} \arrow[d, "i^{*}\eta"]\\
i^{*}j_{*}\Sigma^{\infty}_{+}X_{K} \arrow[r, "i^{*}j_{*}j^{*}f_{D}"] & i^{*}j_{*}\Sigma^{2n,n}H\mathbb{F}_{p}^{K}
\end{tikzcd}
\end{equation}

From Theorem \ref{T:phiproperties}, $P^{m}_{k}=\Phi(P^{m}_{K})=\pi \circ i^{*}j_{*}P^{m}_{K}\circ i^{*}\eta$. Hence, from diagram \ref{diag1special}, we get that $$\pi \circ i^{*}\eta\circ P^{m}_{k} \circ i^{*}f_{D}=\pi \circ i^{*}j_{*}P^{m}_{K}\circ i^{*}j_{*}j^{*}f_{D}\circ i^{*}\eta$$ in the following commuting diagram.

\begin{equation} \label{diag2special}
\begin{tikzcd}
\Sigma^{\infty}_{+}X_{k} \arrow[r, "i^{*}f_{D}"] \arrow[d, "i^{*}\eta"] & \Sigma^{2n,n}H\mathbb{F}_{p}^{k} \arrow[d, "i^{*}\eta"] \arrow[r, "P^{m}_{k}"] & \Sigma^{2(n+m(p-1)),n+m(p-1)}H\mathbb{F}_{p}^{k} \arrow[d, "i^{*}\eta"] \\
i^{*}j_{*}\Sigma^{\infty}_{+}X_{K} \arrow[r, "i^{*}j_{*}j^{*}f_{D}"] & i^{*}j_{*}\Sigma^{2n,n}H\mathbb{F}_{p}^{K} \arrow[r, "i^{*}j_{*}P^{m}_{K}"] & i^{*}j_{*}\Sigma^{2(n+m(p-1)),n+m(p-1)}H\mathbb{F}_{p}^{k} \arrow[d, "\pi"] \\
& & \Sigma^{2(n+m(p-1)),n+m(p-1)}H\mathbb{F}_{p}^{k}
\end{tikzcd}
\end{equation}

\indent Write $P^{m}_{K}(\left<Z_{K}\right>)=\sum_{l=1}^{q}a_{l}\left<Z^{l}_{K}\right>$ for some $q, a_{l} \in \mathbb{Z}$ and closed subvarieties $Z^{l}_{K} \subset X_{K}$ of codimension $n+m(p-1)$. Taking the associated reduced closed subschemes in $X$, we get an element $\sum_{l=1}^{q}a_{l}\left<\overline{Z}^{l}_{K}\right> \in H^{2(n+m(p-1)),n+m(p-1)}(X, \mathbb{F}_{p})$ which corresponds to a morphism $g:\Sigma^{\infty}_{+}X \to \ \Sigma^{2(n+m(p-1)),n+m(p-1)}\widehat{H}\mathbb{F}_{p}^{D}$. For $1 \leq l \leq q$, let $Z^{l}_{k}$ denote the special fiber of $\overline{Z}^{l}_{K}$. Taking pullbacks, $i^{*}g$ gives $\sum_{l=1}^{q}a_{l}\left<Z^{l}_{k}\right> \in H^{2(n+m(p-1)),n+m(p-1)}(X_{k}, \mathbb{F}_{p})$ and $j^{*}g=\sum_{l=1}^{q}a_{l}\left<Z^{l}_{K}\right>=P^{m}_{K}(\left<Z_{K}\right>)$.  Applying $i^{*}\eta$ to $g$, we get a commuting diagram. 

\begin{equation} \label{diagspecial3}
\begin{tikzcd}
\Sigma^{\infty}_{+}X_{k} \arrow[r, "i^{*}g"] \arrow[d, "i^{*}\eta"] & \Sigma^{2(n+m(p-1)),n+m(p-1)}H\mathbb{F}_{p}^{k} \arrow[d, "i^{*}\eta"] \\
i^{*}j_{*}\Sigma^{\infty}_{+}X_{K} \arrow[r, "i^{*}j_{*}j^{*}g"] & i^{*}j_{*}\Sigma^{2(n+m(p-1)),n+m(p-1)}H\mathbb{F}_{p}^{K} \arrow[d, "\pi"] \\ & \Sigma^{2(n+m(p-1)),n+m(p-1)}H\mathbb{F}_{p}^{k}
\end{tikzcd}
\end{equation}

\indent From diagrams \ref{diag2special} and \ref{diagspecial3}, we get $$i^{*}g=\sum_{l=1}^{q}a_{l}\left<Z^{l}_{k}\right>=\pi \circ i^{*}j_{*}j^{*}g \circ i^{*}\eta=\pi \circ i^{*}j_{*}(P^{m}_{K}(\left<Z_{K}\right>)) \circ i^{*}\eta\\$$
$$=\pi \circ i^{*}j_{*}P^{m}_{K} \circ i^{*}j_{*}j^{*}f_{D} \circ i^{*}\eta=P^{m}_{k}(\left<Z_{k}\right>)$$

as required.
\end{proof}

\indent We recall some facts about flag varieties, using \cite{Koc} as a reference. Let $G_{k}$ be a split reductive group over $k$ with Borel subgroup $B_{k}$ and Weyl group $W$. From the Bruhat decomposition, we have  $$G_{k}/B_{k}=\coprod _{w \in W} B_{k}wB_{k}/B_{k}.$$ For $w \in W$, the closure $X^{w}_{k}$ of $B_{k}wB_{k}/B_{k}$ in $G_{k}/B_{k}$ is called a Schubert variety and $$B_{k}wB_{k}/B_{k} \cong \mathbb{A}^{l(w)}_{k}$$ where $l(w)$ is the length of $w$ in $W$. Let $P_{k} \supseteq B_{k}$ be a parabolic subgroup of $G_{k}$. We have $P_{k}=BW_{P}B$ for some subgroup $W_{P} \leq W$.  There is a related $W^{P} \subset W$, such that for each $w \in W^{P}$, $B_{k}wB_{k}/B_{k}$ is isomorphic to $B_{k}wB_{k}/P_{k}$  under the quotient morphism $G_{k}/B_{k} \to G_{k}/P_{k}$ \cite[Lemma 1.2]{Koc}. We also have a cell decomposition $$G_{k}/P_{k}=\coprod _{w \in W^{P}} B_{k}wB_{k}/P_{k}.$$ This cell decomposition is independent of the field $k$. It follows that the total chow group $CH^{*}(G_{k}/P_{k})$ is freely generated as an additive group by the cycle classes $\left<Y^{w}_{k}\right>$ of the images $Y^{w}_{k}$ of the Schubert varieties $X^{w}_{k}$ for $w \in W^{P}$. 

\indent Chevalley \cite{Chev} and Demazure \cite{Dem} showed that the chow rings $$CH^{*}(G_{F_{1}}/P_{F_{1}})\; \textup{and} \; CH^{*}(G_{F_{2}}/P_{F_{2}})$$ are isomorphic for any two fields $F_{1}, F_{2}$. The isomorphism is given by mapping the class of a Schubert subscheme $Y^{w}_{F_{1}}$ to $Y^{w}_{F_{2}}$ for $w \in W^{P}$.  We now prove that the Steenrod operations $P^{n}_{k}$ and $P^{n}_{K}$ give the same action on $H^{2*,*}(G_{k}/P_{k}, \mathbb{F}_{p}) \cong CH^{*}(G_{k}/P_{k})/p \cong CH^{*}(G_{K}/P_{K})/p \cong H^{2*,*}(G_{K}/P_{K}, \mathbb{F}_{p}).$

\begin{corollary} \label{P:flagvar} Let $n \geq 0$ and let $w_{0} \in W^{P}$. We have $$P^{n}_{K}(\left<Y^{K}_{w_{0}}\right>)=\sum_{w \in W^{P}} a_{w}\left<Y^{K}_{w}\right>$$ in $CH^{*}(G_{K}/P_{K})/p$ for some $a_{w} \in \mathbb{Z}$. Then $$P^{n}_{k}(\left<Y^{k}_{w_{0}}\right>)=\sum_{w \in W^{P}} a_{w}\left<Y^{k}_{w}\right>.$$
\end{corollary}

\begin{proof} 
We refer to \cite{Con} for facts about integral models of split reductive groups. Let $w\in W$ and let $X^{w}_{D}$ be the reduced closed subscheme of $G_{D}/B_{D}$ associated to $B_{D}wB_{D}/B_{D}$. Note that $X^{w}_{D}$ is flat over $\mathrm{Spec}(D).$ For any field $F$ and morphism $\mathrm{Spec}(F) \to \mathrm{Spec}(D)$, the fiber $X^{w}_{D} \times _{\mathrm{Spec}(D)} \mathrm{Spec}(F)$ in $G_{F}/B_{F}$ is isomorphic to $X^{w}_{F}$ \cite[Theorem 2]{Ses}. The main point to check is that the fibers of $X^{w}_{D}$ over $\mathrm{Spec}(D)$ are reduced.

\indent Now assume that $w\in W^{P}$. Let $Y^{w}_{D}$ denote the image of $X^{w}_{D}$ in $G_{D}/P_{D}$. Then $Y^{w}_{D} \times _{\mathrm{Spec}(D)} \mathrm{Spec}(F) \cong Y^{w}_{F}$ for any field $F$ and morphism $\mathrm{Spec}(F) \to \mathrm{Spec}(D)$. Proposition \ref{P:specialization} then applies to finish the proof.

\end{proof}

\section{Applications to quadratic forms} \label{S:appquad}

In this section, we use the Steenrod squares $\textrm{Sq}^{2n}_{k}$ to prove new results about nonsingular quadratic forms over a field $k$ of characteristic $2$. The results we prove have analogues in characteristic $\neq 2$ conveniently found in \cite[Sections 79-82]{EKM} where the only missing ingredient for extending to characteristic $2$ was the existence of Steenrod squares satisfying expected properties.

\indent Recall that a quadratic form $(q,V)$ over $k$ is nonsingular if the associated radical $V^{\perp}$ is of dimension at most $1$ and $q$ is nonzero on $V^{\perp}\setminus{0}$. Equivalently, $(q,V)$ is nonsingular if the associated projective quadric is smooth. Note that nonsingular quadratic forms are called nondegenerate in \cite{EKM}. In characteristic $2$, anisotropic quadratic forms are not necessarily nonsingular. Let $(q, V)$ be a nonsingular anisotropic quadratic form of dimension $D$ defined over $k$ and let $X$ be the associated quadric. Over some field extension $F$ of $k$, the quadric $X_{F}$ becomes split. A computation of $CH^{*}(X_{F})$ can be found in \cite[Chapter XIII]{EKM}. Let $h \in CH^{1}(X_{F})$ denote the pullback of the hyperplane class in $\mathbb{P}(V)$ and let $l_{d} \in CH_{d}(X_{F})$ denote the class of a $d$-dimensional subspace in $X_{F}$ where $d= \lfloor (D-1)/2 \rfloor$. Let $l_{i}=h^{i} \cdot l_{d}$ for $ 0 \leq i \leq d$.

\begin{proposition}
As an additive group, $CH^{*}(X_{F})$ is freely generated by $h^{i}, l_{i}$ for $ 0 \leq i \leq d$. For the ring structure, $h^{d+1}=2l_{D-d-1}$, $l^{2}_{d}=0$ if $4$ does not divide $D$, and $l^{2}_{d}=l_{0}$ if $4$ divides $D$.
\end{proposition}

\indent From Corollary \ref{P:flagvar}, the action of the Steenrod squares $\textrm{Sq}^{2n}_{F}$ on $CH^{*}(X_{F})/2$ agrees with the action of Steenrod squares on the mod $2$ Chow ring of a split quadric in characteristic $0$. We refer to \cite[Corollary 78.5]{EKM} for the calculation of the action of Steenrod squares on the mod $2$ Chow ring of a split quadric in characteristic $0$.

\begin{proposition} \label{P:action on quad}
For any $0 \leq i \leq d$ and $j \geq0$, $$\mathrm{Sq}^{2j}_{F}(h^{i})=\binom{i}{j}h^{i+j} \, \,and \, \, \,
      \mathrm{Sq}^{2j}_{F}(l_{i})=\binom{D+1-i}{j}l_{i-j}.$$
\end{proposition}

\indent To state our results, we recall the definition of relative higher Witt indices. Let $\varphi$ be a nonsingular quadratic form over a field $F$ and let $F(\varphi)$ denote the function field of the associated quadric. Let $\varphi_{an}$ denote the anisotropic part of $\varphi$ and let $\frak{i}_{0}(\varphi)$, the Witt index of $\varphi$, denote the dimension of a maximal isotropic subspace for $\varphi$. Start with $\varphi_{0}\coloneqq \varphi_{an}$ and $F_{0} \coloneqq F$. Inductively define $F_{i} \coloneqq F_{i-1}(\varphi_{i-1})$ and $\varphi_{i} \coloneqq (\varphi_{F_{i}})_{an}$ for $i>0$. There exists an integer $\frak{h}(\varphi)$ which is called the height of $\varphi$ such that $\mathrm{dim}\varphi_{\frak{h}(\varphi)} \leq 1$. For $1\leq j \leq \frak{h}(\varphi)$, we then define the $j$th relative higher Witt index $\frak{i}_{j}(\varphi)$ to be $\frak{i}_{0}(\varphi_{F_{j}})-\frak{i}_{0}(\varphi_{F_{j-1}})$. 

\indent We recall Hoffmann's conjecture on the possible values of the first Witt index of an anisotropic quadratic form. Hoffmann's conjecture was originally restricted to quadratic forms over a field of characteristic $\neq 2$ but it makes sense to consider the conjecture in characteristic $2$ as well. For an integer $n$, let $v_{2}(n)$ denote the $2$-adic exponent of $n$.

\begin{conjecture}
Let $\varphi$ be an anisotropic quadratic form over a field $F$ such that $\mathrm{dim}\varphi \geq 2$. Then $\frak{i}_{1}(\varphi) \leq 2^{v_{2}(\mathrm{dim}\varphi-\frak{i}_{1}(\varphi))}.$
\end{conjecture}

\indent Hoffmann's original conjecture for characteristic $\neq 2$ was proved by Karpenko in \cite{Kar}. Karpenko's proof makes use of Steenrod squares on mod $2$ Chow groups. With our construction of Steenrod squares on the mod $2$ Chow groups over a base field of characteristic $2$, we can now prove Hoffmann's conjecture for nonsingular anisotropic quadratic forms over a field of characteristic $2$.

\begin{proposition} \label{P:newquad}
Let $\varphi$ be a nonsingular anisotropic quadratic form over $k$ such that $\mathrm{dim}\varphi \geq 2$. Then $\frak{i}_{1}(\varphi) \leq 2^{v_{2}(\mathrm{dim}\varphi-\frak{i}_{1}(\varphi))}.$
\end{proposition}
\begin{proof}
The proof of \cite[Proposition 79.4]{EKM} works in this case and uses the computation of the Steenrod squares on the mod $2$ Chow ring of a split quadratic given by Proposition \ref{P:action on quad} along with Corollary \ref{corbasechange} on base change of the Steenrod squares. From the Cartan formula \ref{sectioncartan} and results on shell triangles in \cite[Sections 72,73]{EKM} that were proved in arbitrary characteristic, we see that the conclusion of \cite[Lemma 79.3]{EKM} holds for nonsingular anisotropic quadratic forms in characteristic $2$.
\end{proof}

\indent Proposition \ref{P:newquad} provides further evidence for the validity of Hoffmann's conjecture in characteristic $2$. Scully has proved that Hoffmann's conjecture is valid for totally singular quadratic forms over a field of characteristic $2$ \cite{Scu}.  

\indent To finish, we extend $3$ more results of Karpenko on quadratic forms in characteristic $ \neq 2$ to the case of nonsingular anisotropic quadratic forms in characteristic $2$. Let $\varphi$ be a nonsingular anisotropic quadratic form defined over a field $k$ of characteristic $2$ with relative higher Witt indices $\frak{i}_{j} \coloneqq \frak{i}_{j}(\varphi)$ as defined above for $j=1, \ldots, \frak{h}\coloneqq \frak{h}(\varphi).$

\begin{proposition}
Assume that $\frak{h}>1$. Then $$v_{2}(\frak{i}_{1}) \geq \textup{min}(v_{2}(\frak{i}_{2}), \ldots, v_{2}(\frak{i}_{\frak{h}}))-1.$$
\end{proposition}
 \begin{proof} The analogue of this proposition in characteristic $\neq2$ can be found in \cite[Corollary 81.19]{EKM}. The proof of \cite[Corollary 81.19]{EKM} works over a base field of characteristic $2$ using the properties we have established for the Steenrod squares $\textup{Sq}^{2n}_{k}.$ The conclusions of \cite[Lemma 80.1]{EKM} and \cite [Theorem 80.2]{EKM} hold in our situation since $\textup{Sq}^{2n}_{k}$ acts by squaring on $CH^{n}(-)/2$ by Proposition \ref{P:pthpower} and the total homological Steenrod square commutes with proper pushforward by Proposition \ref{pushforward}.
\end{proof}

\indent We next discuss the characteristic $2$ analogue of the ``holes in $I^n$" result \cite[Corollary 82.2]{EKM}. For a field $F$, the quadratic Witt group $I_{q}(F)$ is defined as the quotient of the Grothendieck group of the monoid of isometry classes of even-dimensional nonsingular quadratic forms by the subgroup generated by the hyperbolic plane \cite[Section 8]{EKM}. There is an action of the Witt ring $W(F)$ of nondegenerate symmetric bilinear forms on $I_{q}(F)$. Let $I(F) \subset W(F)$ denote the fundamental ideal of $W(F)$ and set $I_{q}^{n}(F) \coloneqq I^{n-1}(F) \cdot I_{q}(F)$ for $n \geq 1$. Let $k$ be a field of characteristic $2$. Mimicking the proof \cite[Corollary 82.2]{EKM} with $I_{q}^{n}(k)$ used in place of $I^{n}(k)$, we get the following result. Let $n \geq 1$.

\begin{proposition} Let $\varphi \in I_{q}^{n}(k)$ be a nonsingular anisotropic quadratic form such that $\textup{dim} \varphi < 2^{n+1}.$ Then there exists $  0 \leq i \leq n$ such that $\textup{dim} \varphi=2^{n+1}-2^{i+1}.$

\end{proposition}

\indent Our last result concerns $u$-invariants of fields. Following \cite[Section 36]{EKM}, the $u$-invariant $u(F)$ of a field $F$ is defined to be the smallest non-negative integer (or $\infty$ if there is no such integer) $u(F)$ such that every nonsingular locally hyperbolic quadratic form $\varphi$ over $F$ with $\textup{dim}\varphi>u(F)$ is isotropic. Over a field of finite characteristic, every quadratic form is locally hyperbolic. 

\indent In \cite{Vis}, Vishik constructed characteristic $0$ fields of $u$-invariant $2^{r}+1$ for all $r \geq 3$. Karpenko used Steenrod squares on mod $2$ Chow groups to show that for any $r \geq 3$ and any field $F$ of characteristic $\neq 2$, $F$ is contained in a field of $u$-invariant $2^{r}+1$ \cite{Kar2}. Karpenko's constructions in \cite{Kar2} now extend to fields of characteristic $2$ through the use of the Steenrod squares $\textup{Sq}^{2n}_{k}$ defined in this paper for $k$ of characteristic $2$.

\begin{proposition} Let $k$ be a field of characteristic $2$ and let $r \geq 3$. Then $k$ is a subfield of a field of $u$-invariant $2^{r}+1.$
\end{proposition}

\end{document}